\documentclass[12pt]{amsart}
\usepackage{amscd, amssymb}
\input{xy}
\xyoption{all}

\newtheorem{Thm}{Theorem}[subsection]

\newtheorem{Prop}[Thm]{Proposition}
\newtheorem{Def}[Thm]{Definition}
\newtheorem{Def/Thm}[Thm]{Definition/Theorem}
\newtheorem{Cor}[Thm]{Corollary}
\newtheorem{Lemma}[Thm]{Lemma}
\newtheorem{Ass}[Thm]{Assumption}

\theoremstyle{remark}
\newtheorem{Rmk}[Thm]{Remark}

\newtheorem{EG}[Thm]{Example}
\numberwithin{equation}{subsection}

\newcommand{\ot }{\otimes}
\newcommand{\ra }{\rightarrow}

\newcommand{\lra }{\longrightarrow}
\newcommand{\Map}{{\mathrm{Map}}}

\newcommand{\Hom }{{\mathrm{Hom}}}

\newcommand{\Spec}{{\mathrm{Spec}}}

\newcommand{\Pic}{{\mathrm{Pic}}}

\newcommand{\cA}{{\mathcal{A}}}
\newcommand{\cO}{{\mathcal{O}}}

\newcommand{\cL}{{\mathcal{L}}}
\newcommand{\cE}{{\mathcal{E}}}
\newcommand{\cF}{{\mathcal{F}}}
\newcommand{\cH}{{\mathcal{H}}}
\newcommand{\cP}{{\mathcal{P}}}
\newcommand{\cQ}{{\mathcal{Q}}}
\newcommand{\cV}{{\mathcal{V}}}
\newcommand{\cC}{{\mathcal{C}}}

\newcommand{\cX}{{\mathcal{X}}}

\newcommand{\fC}{{\mathfrak{C}}}
\newcommand{\fP}{{\mathfrak{P}}}
\newcommand{\G}{{\bf G}}
\newcommand{\bT}{{\bf T}}
\newcommand{\bB}{{\bf B}}
\newcommand{\bS}{{\bf S}}

\newcommand{\PP }{{\mathbb P}}

\newcommand{\QQ }{{\mathbb Q}}
\newcommand{\CC }{{\mathbb C}}
\newcommand{\ZZ }{{\mathbb Z}}
\newcommand{\RR }{{\mathbb R}}

\newcommand{\Mgk}{\overline{M}_{g,k}}
\newcommand{\cCgk}{\mathcal{C}_{g,k}}
\newcommand{\fMgk}{\mathfrak{M}_{g,k}}
\newcommand{\fCgk}{\mathfrak{C}_{g,k}}
\newcommand{\fBun}{\mathfrak{B}un_{\bf {G}}}

\newcommand{\fS}{\mathfrak{S}}
\newcommand{\VmodG}{V/\!\!/\G}
\newcommand{\WmodG}{W/\!\!/\G}
\newcommand{\QmapV}{\mathrm{Qmap}_{g,k}(\VmodG,\beta)}
\newcommand{\QmapW}{\mathrm{Qmap}_{g,k}(\WmodG,\beta)}
\newcommand{\QmapWe}{\mathrm{Qmap}_{g,k}^\epsilon(\WmodG,\beta)}
\newcommand{\BunG}{\mathfrak{B}un_\G}
\newcommand{\BunGbeta}{\mathfrak{B}un_{\G,\beta}}

\begin{document}
\title{Stable quasimaps to GIT quotients}

\begin{abstract} We construct new compactifications with good properties of moduli spaces
of maps from nonsingular marked curves to a large class of GIT quotients. This generalizes from a unified
perspective many particular examples considered earlier in the literature.
\end{abstract}
\dedicatory {Dedicated to the memory of Professor Hyo Chul Myung}

\author{Ionu\c t Ciocan-Fontanine}
\noindent\address{School of Mathematics, University of Minnesota, 206 Church St. SE,
Minneapolis MN, 55455, and\hfill
\newline \indent School of Mathematics, Korea Institute for Advanced Study,
85 Hoegiro, Dongdaemun-gu, Seoul, 130-722, Korea}
\email{ciocan@math.umn.edu}

\author{Bumsig Kim}
\address{School of Mathematics, Korea Institute for Advanced Study,
85 Hoegiro, Dongdaemun-gu, Seoul, 130-722, Korea}
\email{bumsig@kias.re.kr}

\author{Davesh Maulik}
\address{Department of Mathematics, Massachusetts Institute of Technology, Cambridge, MA }
\email{dmaulik@math.mit.edu}

\maketitle

\section{Introduction}

The aim of this paper is to introduce new
virtually smooth modular compactifications of spaces
of maps from curves to a large class of targets.

Let $C$ be a nonsingular complex projective curve, let $X$ be a nonsingular complex projective variety, and let $\beta\in H_2(X,\ZZ)$
be an effective curve class. The standard way of putting a scheme structure on the set $\mathrm{Map}_\beta(C,  X)$ of algebraic maps
from $C$ to $X$ of class $\beta$ is by viewing it as an open subset in the Hilbert scheme of the product $C\times X$,
via the identification of a map with its graph.
The restriction of the universal family on the Hilbert scheme gives a universal map
$$\begin{array}{ccc}
C\times\mathrm{Map}_\beta(C,X)&\stackrel f\lra & X\\ \\
\downarrow \pi& & \\ \\
\mathrm{Map}_\beta(C,X)& &
\end{array}
$$
The obstruction theory of the Hilbert scheme restricts to the obstruction theory
\begin{equation}\label{obstheory}\left (R^\bullet\pi_*f^*T_X\right)^\vee\end{equation}
of the space of maps. While \eqref{obstheory} is a perfect obstruction theory, and therefore $\mathrm{Map}_\beta(C,X)$ is
virtually smooth, it is well-known
that the obstruction theory of the Hilbert scheme fails to be perfect at points corresponding
to subschemes which are not local complete intersections. Hence the compactification provided by the Hilbert scheme
is not suitable to defining invariants of $X$ via intersection theory. For a few special
cases of varieties $X$, other compactifications have been used which do not suffer from this
defect. Specifically,

$(i)$ when $X$ is a Grassmannian, the
Quot schemes on $C$ carry a perfect obstruction theory, see \cite{MO} and \cite{CiKa2};

$(ii)$ when $X$ is a toric variety, the ``toric compactifications" (or linear sigma models) of Givental and Morrison-Plesser
also carry perfect obstruction theories;

$(iii)$ when $X$ is the Hilbert scheme of points in $\CC^2$, Diaconescu's moduli spaces of ADHM sheaves, \cite{D}, give partial
compactifications carrying perfect obstruction theories.

All these spaces have the common feature that the curve $C$ is kept fixed and the map degenerates, so that 
points in the boundary correspond to {\it rational} maps $C - -\ra X$.

In the case of a general $X$, Kontsevich approached the problem by insisting that in the compactification the boundary
points still parametrize honest maps. This requires that the domain curve degenerates. It is then natural also
to allow the domain curve to vary in moduli (and to consider markings). The approach leads to the
virtually smooth moduli stacks of stable maps $\Mgk(X,\beta)$ used to define the Gromov-Witten invariants of $X$.
For stable map spaces, (\ref{obstheory}) becomes a relative perfect obstruction theory over the moduli Artin stack
$\fMgk$ of (prestable) domain curves. In particular, the graph space $\overline{M}_{g,0}(X\times C, (\beta, 1))$ provides
a virtually smooth compactification of $\mathrm{Map}_\beta(C,X)$. 

The paper \cite{MOP} introduced, in case $X$ is a Grassmannian, a hybrid version of these compactifications,
called the moduli of stable quotients. It allows both the domain curve to vary, and the map to acquire base points. Then
\cite{MOP} shows that after imposing an appropriate stability condition, the resulting moduli space is a proper Deligne-Mumford
stack of finite type, carrying a relative perfect obstruction theory over $\fMgk$ and hence a virtual fundamental class.

The paper \cite{CK} treated the toric counterpart of the story, that is, the hybrid version of the compactifications in $(ii)$ above
when the domain curve is allowed to degenerate and to carry markings. The resulting objects, called {\it stable toric quasimaps},
are proven there to form a proper DM stack of finite type.
Furthermore, it is observed that it is more natural to
view the moduli spaces as stacks over (products of) {\it Picard stacks over} $\fMgk$, the corresponding
relative obstruction theory being that of sections of line bundles on curves.

This suggests that the common feature
of examples $(i)$-$(iii)$ which should be exploited is that they are all GIT quotients. The obstruction theory
that needs then to be studied is that of sections of fibrations associated to principal $\G$-bundles over curves, with $\G$ a reductive 
complex algebraic
group.  Once this is recognized, the most natural way to view the problem is in a stack-theoretic context.

 Let $X=\WmodG$ be a (nonsingular, projective) GIT quotient obtained from a linearized $\G$-action on a quasiprojective
 variety $W$ for which the stable and semistable loci coincide. Let $\cA=[W/\G]$ be the {\it stack quotient}. A map from a
 prestable curve $C$ to $\cA$ corresponds to a pair $(P,u)$, with $P$ a principal $\G$-bundle on $C$ and $u$ a section of the
 induced fiber bundle $P\times_\G W\lra C$ with fiber $W$. It has a ``homology" class $\beta\in \Hom_\ZZ(\Pic(\cA),\ZZ)$.
 We then have the moduli stack $\fMgk([W/\G],\beta)$ parametrizing
families of tuples
$$((C,p_1,\dots,p_k),P,u)$$
with $(C,p_1,\dots,p_k)$ a prestable $k$-pointed curve of genus $g$. While $\fMgk([W/\G],\beta)$ contains the Deligne-Mumford stack
$$M_{g,k}(\WmodG,\beta)$$
of maps from nonsingular pointed curves to $\WmodG$ as an open substack,
it is in general a nonseparated Artin stack of infinite type.

A sophisticated study of stacks  of the type $\fMgk([W/\G],\beta)$ was begun in \cite{FTT},
with the goal to produce an algebro-geometric version of the ``Gromov-Witten gauge invariants"
introduced in symplectic geometry via solutions to vortex equations (see \cite{CGS00, CGMiRS02, MiR03, MiRT04, GW08, GW09}). 
The point of view we pursue
in this paper is different in that we seek to obtain moduli stacks with good properties by imposing
additional conditions on the maps we consider, as well as on the variety $W$. Namely, we will want our moduli stacks to be
open substacks in $\fMgk([W/\G],\beta)$  which are of finite type, Deligne-Mumford, proper, and to carry a perfect obstruction theory extending the
natural such obstruction theory on the stack $M_{g,k}(\WmodG,\beta)$.
This will allow us
to construct systems of numerical ``invariants" (in fact, new Cohomological Field Theories) on the cohomology of the target varieties.

We highlight here the main points:
\begin{itemize}
\item If we restrict to $((C,p_1,\dots,p_k),P,u)$ such that $u$ sends the generic point of each irreducible component of $C$ to the
GIT-{\it stable locus} $W^s\subset W$, the resulting moduli stack is of finite type over $\fMgk$. We call these $((C,p_1,\dots,p_k),P,u)$ {\it quasimaps to}
$\WmodG$. The closed points $y\in C$ with $u(y)\in W\setminus W^s$ are called the {\it base-points} of the quasimap (we assume throughout that there are no
strictly semistable points in $W$).

\item To ensure separatedness, we restrict to quasimaps whose base-points (if any) are away from the nodes and markings of $C$.

\item To obtain a Deligne-Mumford stack of finite type, we require the quasimaps to satisfy a stability condition.

\item To ensure properness we require $W$ to be {\it affine}.

\item To ensure that the natural obstruction theory is perfect we require that 
$W$ has at most local complete intersection singularities.

\end{itemize}

The stability condition we consider initially generalizes the ones from \cite{MOP}, \cite{CK} and we call the
resulting objects stable quasimaps to $\WmodG$. 
It is a simple fact
that Kontsevich's stable maps to $\WmodG$ are also obtained by imposing a
stability condition of the same kind on quasimaps.
Moreover, generalizing the observation of \cite{MM1}, \cite{MM2} for target $\PP^n$ (extended in \cite{Toda} to Grassmannians),
we show that there is a one-parameter family of stability conditions producing open substacks of $\fMgk([W/\G],\beta)$
with all the good properties we require, and which interpolate between stable maps and stable quasimaps. The virtual classes
of these moduli spaces will potentially differ over loci in their boundaries. At the level of invariants, a correspondence between
the stable quasimap and Gromov-Witten theories of $\WmodG$ will then require a (generally nontrivial) transformation which
records wall-crossing contributions. In this paper we restrict ourselves to setting up the foundations of quasimap theory, leaving the
investigation of the correspondence to future work.

We now describe in more detail the contents and organization of the paper. 

Section 2 is devoted to a description of
the general set-up that will be used for the rest of the paper: we consider nonsingular GIT quotients obtained from the action of  
a reductive complex algebraic group
$\G$ on an affine complex algebraic variety $W=\Spec(A)$, linearized by a character of $\G$. The quotient $\WmodG$ is always projective over
the ``affine quotient" $\Spec(A^\G)$ of $W$ by $\G$, but we do not require the affine quotient to be a point.
We also include here a quick discussion of
the moduli stack $\fBun$ parametrizing principal $\G$-bundles on (varying) prestable marked curves.

In section 3 the notions of quasimaps and stable quasimaps to $\WmodG$ are introduced.
We then prove various boundedness results for families
of quasimaps and stable quasimaps.

The main properties of the moduli stack $\QmapW$ of stable quasimaps to $\WmodG$ 
with fixed numerical data $(g,k,\beta)$ are established
in Section 4. While the fact that $\QmapW$ is an Artin stack locally of finite type follows quite easily
by combining some results in \cite{Lieblich} and \cite{AOV}, we give a straightforward direct
construction which exhibits it as a stack with a schematic morphism to $\fBun$, whose fibers are
open subschemes in appropriate Hilbert schemes. Boundedness is used to show that the stack is of finite type, while stability implies that it is Deligne-Mumford.
Finally, we prove that $\QmapW$ is proper over the affine quotient. Our argument uses crucially the assumption
that $W$ is affine.

The usefulness of the construction of the moduli stack 
described above is that it provides us immediately
with a canonical obstruction theory for $\QmapW$, relative to the smooth stack $\fBun$, given by the deformation theory
of Hilbert schemes.
When $W$ is nonsingular it is very easy to see that this obstruction theory is perfect. By a
more elaborate argument we prove that the obstruction theory remains perfect if lci singularities only
are allowed in the unstable locus of $W$. It is also clear from the description of the obstruction theory
that the result is optimal: as long as the stability condition allows base points, the obstruction theory
will fail to be perfect if the singularities are worse than lci. In particular, this answers a question of Marian, Oprea, and Pandharipande in \cite{MOP}
by showing that the ``moduli of stable quotients" to a projective subvariety $X\subset\PP^n$ carries a natural virtual class only
when $X$ is a complete intersection.
On the other hand, since the stability
condition that leads to the usual Kontsevich stable maps forbids base points,  the unstable locus in $W$ plays no role
and we obtain a perfect obstruction theory
inducing the usual virtual class of $\Mgk(\WmodG,\beta)$ for all targets $\WmodG$.

A short alternative construction is described in Section 5 for the case when $W=V$ is a vector space.
It leads naturally to a simple description of
the {\it absolute} perfect obstruction theory  and the corresponding virtual class of $\mathrm{Qmap}$
is shown to be globally a refined top Chern class of a vector bundle on a smooth Deligne-Mumford stack.
 It is worth pointing out
that in this case both the tangent bundle of $\VmodG$ and the obstruction theory of $\QmapV$ 
may be presented using generalized ``Euler sequences", which facilitate computations in many instances.

Section 6 introduces the descendent quasimap integrals, as well as the analogues in quasimap theory of the
twisted Gromov-Witten invariants of Coates and Givental for proper targets $\WmodG$. Furthermore, 
equivariant invariants for non-proper targets endowed with torus actions having proper fixed points loci 
are also discussed (Nakajima quiver varieties are
a large class of particularly interesting examples of such targets).

Finally, Section 7 deals with two variants of our construction. The first is the variation of stability condition mentioned
above. Namely, we introduce for each positive rational number $\epsilon$ the notion of $\epsilon$-stable
quasimaps. When $\epsilon$ is sufficiently close to zero, this coincides with the notion of stable quasimap, while
when $\epsilon$ is sufficiently large it recovers stable maps to $\WmodG$. The arguments from sections 3 and 4 extend
in a straightforward manner to prove that the moduli stacks of $\epsilon$-stable quasimaps are Deligne-Mumford, of finite type,
proper over the affine quotient, and carry natural perfect obstruction theories under the same condition on the singularities of $W$.
For a fixed class $\beta$, the interval $(0,+\infty)$ over which the stability parameter $\epsilon$ ranges is divided into chambers 
by finitely many walls of the form $\epsilon=1/m$ such that the moduli stack is constant in each chamber.

The second variant we discuss is that of
 stable quasimaps with one parametrized component. These are most
interesting in genus zero and the corresponding moduli stacks
are the analogues for quasimaps of the ``graph spaces" $\overline{M}_{0,k}(X\times\PP^1,(\beta,1))$
used prominently by Givental in his proofs of mirror theorems. They are used to extend the definition of Givental's small
$I$-functions, and of the big $I$-functions in \cite{CK}, from toric targets to general GIT quotient targets $\WmodG$.
We expect there is a natural way to express the ``mirror maps" transforming $I$ into the $J$-function of $\WmodG$
via wall-crossing for $\epsilon$-stable quasimaps with one parametrized component.

\subsection{Acknowledgments} During the preparation of the paper
we benefited from conversations with Daewoong Cheong, Max Lieblich, Rahul Pandharipande, and Jason Starr. 
Ciocan-Fontanine thanks KIAS for financial support, excellent
working conditions, and an inspiring research environment.
I.C.-F. was partially supported by the NSF grant DMS-0702871, and the NSA grant
H98230-11-1-0125.
B.K. was supported by the KRF grant 2007-341-C00006. D.M. was supported by a Clay Research Fellowship.

\section{Preliminaries} Throughout the paper we work over the base field $\CC$. We fix a reductive (hence
linearly reductive) complex algebraic
group $\G$.

Recall that a principal $\G$-bundle $P$ over a base
algebraic space $B$ is an algebraic space $P$ with a free (right) $\G$-action and a $\G$-equivariant
map $\pi: P\lra B$ (where $B$ has the trivial action) which is {\' e}tale-locally trivial.

For a scheme $W$ with (right) $\G$-action, one can form the
{\it mixed construction} 
\begin{equation}\label{mixed}P\times_\G W:=[(P\times W)/\G],
\end{equation}
where the quotient is taken with respect to the diagonal action of $\G$.
It is an algebraic space with a morphism 
$$\rho : P\times_\G W\lra B$$
which is  (in the {\' e}tale topology) a locally trivial fibration
with fiber $W$.

\subsection{Principal $\G$-bundles on curves}
By $\fMgk$ we denote the moduli stack of prestable $k$-pointed curves of genus $g$, and by $\Mgk$ we denote  the moduli
stack of stable $k$-pointed curves of genus $g$. Recall that $\fMgk$ is a smooth Artin stack, locally of finite type, while $\Mgk$ is
a smooth Deligne-Mumford stack of finite type, which is proper.
Usually $\fCgk$ and $\cCgk$ will stand for the universal curves over $\fMgk$ and $\Mgk$, respectively. The projection maps from
the universal curves to their base will all be denoted $\pi$; this should not lead to confusion.

We will also consider the relative moduli stack
$$\fBun\stackrel\phi\lra\fMgk$$
of principal $\G$-bundles on the fibers of $\fCgk$.
It is again an Artin stack, locally of finite type, as can be deduced immediately from \cite[Prop. 2.18]{Lieblich} and \cite[Lemma C.5]{AOV}). By \cite{Behrend}, the morphism $\phi$ has smooth geometric fibers. In fact,
$\phi$ is a smooth morphism, hence $\fBun$ is smooth. 
For convenience, we include a direct proof of these facts in the following proposition. See also the recent preprint \cite{Wang}.

%%%%%%%%%%%%%%%%%%%%%%%%%%%%%%%%%%%%%%%%%%%%%%%%%%%
\begin{Prop}\label{bunG}
The stack $\fBun$ is a smooth Artin stack, locally of finite type over $\CC$.
\end{Prop}

\begin{proof}
We check the requirements in the definition of an Artin stack.

1. The diagonal map $\Delta$ is representable, quasicompact, and separated:

a) We first prove this for the case $\G=GL(r)$. Let $T$ be a $\CC$-scheme.
Consider two vector bundles $E_1, E_2$ of rank $r$ on a projective scheme $Y/T$ with $\pi _{Y/T}$ the flat structural map $Y\ra T$. 
We let $\mathrm{Isom}_{T}(E_1, E_2)$ be the functor from the $T$-scheme category $\mathrm{Sch} /T$ 
to the set category $\mathrm{Sets}$ sending
$T'/T$ to $\mathrm{Hom}_{T'}(E_1|_{T'}, E_2|_{T'})$.
Then the stack $\mathrm{Isom}_{T}(E_1, E_2)$ over $T$
is representable by a locally closed subscheme of the total space of the vector bundle associated to the locally
free sheaf $(\pi _{Y/T})_*(E_1^\vee\ot E_2 (n))$ for some large enough integer $n$. Hence, $\mathrm{Isom}_{T}(E_1, E_2)$
is a quasicompact separated $T$-scheme.
Now consider two vector bundles $F_1, F_2$ on two families $C_i/S$ of $m$-marked {\em stable} curves over a $\CC$-scheme $S$, respectively.
Notice that $\mathrm{Isom}_S(C_1, C_2)$ is a quasi-projective scheme over $S$.
The stack $\mathrm{Isom}_S(F_1, F_2)$ over $S$ is representable by a quasicompact separated $T$-scheme $\mathrm{Isom}_{T}(\pi _1^*F_1, f^*F_2)$, 
where $T= \mathrm{Isom}_S(C_1, C_2)$,
$Y=C_1\times _S T$, $\pi _1: Y\ra C_1$ is the first projection,
and $f$ is the evaluation map $Y \ra C_2$.  However, in general, $C_i$ are families of prestable curves.
In this case, by adding additional local markings, there is an \'etale covering $\{S_j\}$ of $S$ such that 
$C_{ij}=C_i|_{S_j}$ is a projective scheme over $S_j$. Now glue $\mathrm{Isom}_{S_j}(F_1|_{S_j}, F_2|_{S_j})$ to obtain an algebraic space over $S$, representing
$\mathrm{Isom}_S(F_1, F_2)$. This shows that $\Delta$ is representable, quasicompact, and separated.

b) For a general reductive algebraic group $\G\subset GL(r)$, consider extensions $P_i':= P_i\times _\G GL(r)$, $i=1,2$,
and the associated sections $s_i : C \ra P_i'/\G$.
Then $\mathrm{Isom}_S(P_1, P_2)$ is the closed algebraic subspace of $\mathrm{Isom}_S(P_1', P_2')$ parametrizing
the isomorphisms compatible with $s_i$.

2. There is a smooth surjective morphism from a $\CC$-scheme locally of finite type to $\fBun$:

   Let $S$ be a $\CC$-scheme of finite type.  
   For every pair of positive integers $n, N$ and every family $C/S$ of prestable curves with relatively ample line bundle $\mathcal{O}_C(1)$, consider
   an open subscheme $Z_{N, n}$ of Quot schemes parameterizing the locally free rank $r$ quotients of $\mathcal{O}_C^{\oplus N}(n)$.
   Then when $\G =GL(r)$ the natural forgetful morphism from the disjoint union $Z:=\coprod _{N, n} Z_{N,n}$ 
   to the stack $\fBun (C/S)$ of principal $G$-bundles on $C/S$ is surjective and smooth. 
   For a general reductive algebraic group $\G \subset GL(r)$, we consider
   the natural morphism 
   $$\mathrm{Hom}_{Z} (C\times _S Z , \mathcal{P}/\G)\ra \fBun (C/S),$$
   where $\mathcal{P}$ is the universal $GL(r)$-principal bundle on $C\times _S Z$. 
   Since $GL(r)/\G$ is quasi-projective, and hence
   $\mathcal{P}/\G$ is quasi-projective over $Z$ (see for example \cite[\S 3.6.7]{Sorger}), 
   we see that $\mathrm{Hom}_{Z} (C\times _S Z , \mathcal{P}/\G)$ is
   a quasi-projective scheme over $Z$. 
   Now by this and the fact that $\fMgk$ is an Artin stack, we can build a scheme locally of finite type, smoothly covering $\fBun$.

3. We show the smoothness of $\fBun$. 
It is enough to show that $\fBun$ is formally smooth, i.e., the deformation problem for the corresponding moduli functor is unobstructed: 

Let $S_0$ be the spectrum of a finitely generated $\CC$-algebra $A_0$ and let
$S=\Spec (A)$, with $A$ a square zero extension by a finite $A_0$-module $M$.
 Let  $C_0, C$ be families of prestable curves over $S_0$, respectively over $S$, with $C$ an extension of $C_0$. 
Note that for any principal $\G$ bundle $P'$ on a scheme $Y$,
$\mathrm{Aut}(P') = \mathrm{Hom}_\G (P', \G) = \Gamma (Y, P'\times _\G \G)$, where the $\G$-action on $\G$ is the adjoint action.
From this it is classical to derive that the obstruction to extending a
principal $\G$-bundle on $C_0$ to $C$ lies in
$H^2 (C_0, P\times _\G \mathfrak{g})\ot _{A_0}M$.  The fibers of $C_0/S_0$ are one-dimensional and $S_0$ is affine, hence the $A_0$-module
$H^2 (C_0, P\times _\G \mathfrak{g})$ vanishes. Since $\fMgk$ is a smooth stack over $\CC$, we conclude the proof.
\end{proof}

\subsection{A class of GIT quotients}\label{setup} Let $W=\Spec(A)$ be an affine algebraic variety with an action by the reductive algebraic group $\G$.
There are two natural quotients associated with
this action: the {\it quotient stack} $[W/\G]$, and the {\it affine quotient} $W/_{\mathrm {aff}}\G=\mathrm{Spec}(A^\G)$.

Let $\chi(\G)$
be the character group of $\G$ and
denote by $\Pic^\G(W)$ the group of isomorphism classes of $\G$-linearized line bundles on $W$. Any character
$\xi\in\chi(\G)$ determines a one-dimensional representation $\CC_\xi$ of $\G$, hence
a linearized line bundle $$L_\xi=W\times \CC_\xi$$
in $\Pic^\G(W)$.

Fix once and for all a character $\theta\in\chi(\G)$. Since $\G$ is reductive, the graded algebra
$$S(L_\theta):=\oplus_{n\geq 0}\Gamma(W,L_\theta^{\otimes n})^\G$$
is finitely generated and we have the associated {\it GIT quotient}
$$\WmodG:=W/\!\!/_{\theta}\G:=\mathrm{Proj}(S(L_\theta)),$$
which is a quasiprojective variety, with a projective morphism
$$\WmodG\lra W/_{\mathrm{aff}}\G$$
to the affine quotient (see \cite{King}). The GIT quotient is projective precisely when
$W/_{\mathrm{aff}}\G$ is a point, i.e., when the only $\G$-invariant functions on $W$
are the constant functions.

Let 
$$W^s=W^s(\theta)\;\;\;\;  {\mathrm {and}}\;\;\;\; W^{ss}=W^{ss}(\theta)$$
be the open subsets of  stable (respectively, semistable) points determined 
by the choice of linearization (again, see \cite{King} for the definitions).
From now on the following assumptions will be in force:

$(i)$ $\emptyset\neq W^s=W^{ss}.$

$(ii)$ $W^s$ is nonsingular.

$(iii)$ $\G$ acts freely on $W^s$.

It follows from
Luna's slice theorem that
$W^s\stackrel\rho\lra \WmodG$ is a principal $\G$-bundle in the \'etale topology.
Hence $\WmodG$ is a nonsingular variety, 
which coincides with $[W^s/\G]$, the stack quotient of the stable locus. In particular,
it is naturally an open substack in $[W/\G]$.

Note that $\WmodG$ comes with a (relative) polarization: the line bundle $L_\theta$
descends to a line bundle $$\cO_{\WmodG}(\theta)=W^s\times_\G (L_\theta |_{W^s})$$
which is relatively ample over the affine quotient (the right-hand side is the mixed construction \eqref{mixed}).
Replacing $\theta$ by a positive integer multiple gives the same quotient $\WmodG$, 
but with polarization $\cO_{\WmodG}(m\theta)=\cO_{\WmodG}(\theta)^{\ot m}$. 

\begin{Rmk}\label{orbifold} The assumption $(iii)$ above can in fact be dropped, in which case one considers the smooth Deligne-Mumford open substack
$[W^s/\G] \subset [W/G]$, with coarse moduli space $\WmodG$, as the base of the principal $\G$-bundle. 
This leads to the theory of ``orbifold" (or twisted) stable quasimaps to $[W^s/\G]$, which
can be developed with the same methods we introduce here. For clarity, we decided to treat only the case of manifold targets in this paper. The 
orbifold theory will be discussed in a forthcoming paper, \cite{CCK}.
\end{Rmk}

\subsection{Maps from curves to $[W/\G]$} 
Let $(C,x_1,\dots ,x_k)$ be a prestable pointed curve, i.e., a connected projective curve (of some arithmetic genus $g$),
with at most nodes as singularities, together with $k$ distinct and nonsingular marked points on it.
We will be interested in various moduli spaces of maps from such curves to the quotient stack $[W/\G]$. 
By the definition of quotient stacks, a map $[u]:C\lra [W/\G] $ corresponds to a pair $(P,\tilde{u})$, with 
$$P\lra C$$ 
a principal $\G$-bundle on $C$ and
$$\tilde{u}:P\lra W$$ 
a $\G$-equivariant morphism. Equivalently (and most often), we will consider the data $(P,u)$, with
$$u:C\lra P\times_\G W$$
a section of the fiber bundle $\rho : P\times_\G W\lra C$. Then
$[u]: C \ra [W/\G]$ is obtained as the composite
$$C\stackrel{u}{\ra} P\times _\G W \ra [W/\G].$$
Note also that the composition 
$$C\stackrel{[u]}{\ra} [W/\G]\ra W/_{\mathrm {aff}}\G$$
is always a constant map, since $C$ is projective. 

We denote by $ \Map (C, [W/\G])$ the Artin stack parametrizing such pairs $(P,u)$. 

\subsection{Induced line bundles and the degree of maps} Recall that $L\mapsto [L/\G]$ gives an identification
$\Pic ^\G(W) = \Pic ([W/\G])$.

Let $(C,P,u)$ be as above.
For  a $\G$-equivariant line bundle $L$ on $W$, we have a cartesian diagram
\[ \begin{CD} P\times _\G W @<<< P\times _\G L \\
            @VVV @VVV\\                   
                      [W/\G] @<<< [L/\G] 
                     \end{CD} .\]                
                     Hence we get an induced line bundle $u^*(P\times _\G L) = [u]^*([L/\G])$ on $C$.

 \begin{Def} \label{degree}
The {\it degree} $\beta$ of $(P,u)\in \Map (C, [W/\G])$ 
is the  homomorphism  
$$\beta:\Pic ^{\G}(W) \ra \ZZ,\;\;\;
 \beta(L) =\deg _C(u^*(P\times _\G L)).$$ 
 \end{Def}
 In fact, under the natural map $H_{2}^\G(W)\ra \Hom (\Pic ^\G (W),\ZZ)$, $\beta$ is the image of the class of the equivariant cycle
 $$\begin{CD} P @>>> W\\
 @VVV @.\\
 C @. 
 \end{CD}$$
 given by $(C,P,\tilde{u})$. Here $H_{2}^\G(W)$ denotes the $\G$-equivariant homology.
 As a consequence we will also call $\beta$ the (equivariant homology) {\it class} of $(C,P,u)$.

Finally, we note that any $\G$-equivariant section $t\in \Gamma(W,L)^\G$ induces a section of $[L/\G]$, hence a section $u^*t=[u]^*t$ of
$u^*(P\times _\G L)$.

\subsection{Quotients of vector spaces}\label{vspace1} An important special case is when $W=V$ is a 
finite dimensional $\CC$-vector space equipped with a {\it linear} action, i.e., $\G$ acts via
a representation $\G\lra \mathrm{GL}(V)$. As an algebraic variety, $V$ is an affine space and its Picard group is trivial.
Hence there is an identification
$$\chi(\G)\cong\Pic^\G(V).$$ 

Let $C$ be a nodal curve and let $(P,u)\in \Map (C, [V/\G])$ be as above.
If $\xi$ is a character with associated representation $\CC_\xi$, then there is an isomorphism
of line bundles on $C$
\begin{equation}\label{fiberprod}
u^*(P\times _\G L_\xi)\cong P\times_\G\CC_\xi.
\end{equation}
Therefore in this case the degree 
$$\beta=\beta_P\in \Hom_\ZZ(\chi(\G),\ZZ)$$ is the usual degree of the principal $\G$-bundle on a curve.

\begin{Rmk} Equation \eqref{fiberprod} is a consequence of the following fact, which we will tacitly use from now on.
Given  $\G$-schemes $W$ and $E$ and a map $(C,P,u)$ to $[W/\G]$, let 
$$ \rho :P\times_\G(W\times E)\lra P\times_\G W
$$
be the map induced by the first projection $W\times E\ra W$. Then there is a canonical isomorphism of $C$-schemes
$$P\times_\G E\cong C\times _{P\times_\G W} (P\times_\G(W\times E)),$$
where in the right-hand side we have the fiber product over the section $u:C\ra P\times_\G W$ and the map $\rho$.
\end{Rmk}

The following well-known fact will be used later in the paper. It can be found in many standard texts on Invariant Theory (see e.g. \cite{LeP}, page 94).
\begin{Prop}\label{embedding} Let $W$ be an affine algebraic variety over $\CC$ and let $\G$ be a reductive group acting
on $W$. There exists a finite dimensional vector space $V$ with a linear $\G$-action, and a $\G$-equivariant closed embedding $W\hookrightarrow V$.  
\end{Prop}

\section{Quasimaps and stable quasimaps to $\WmodG$} We keep the set-up from the
previous section: we are given an affine variety $W$ with an action of $\G$ linearized by a 
character $\theta\in\chi(\G)$, satisfying the assumptions $(i)-(iii)$ from \S\ref{setup}. 
\subsection{Quasimaps and stable quasimaps} 

Fix integers $k,g\geq0$ and a class $\beta\in \Hom_\ZZ(\Pic^\G(W),\ZZ)$.

\begin{Def}\label{qmap} A $k$-pointed, genus $g$ {\rm {quasimap}} of class $\beta$ to $\WmodG$ consists of the data

$$( C, p_1,\dots , p_k, P, u),$$
where

\begin{itemize}

\item $(C, p_1,\dots , p_k)$ is a connected, at most nodal, $k$-pointed projective curve of genus $g$,

\item $P$ is a principal $\G$-bundle on $C$,

\item $u$ is a section of the induced fiber bundle $P\times_\G W$ 
with fiber $W$ on $C$ such that $(P,u)$ is of class $\beta$.

\end{itemize}

satisfying the following generic nondegeneracy condition:

\begin{align}\label{generic-ndeg}
&{there\; is\; a\; finite\; (possibly\; empty)\; set}\; B\subset C\;
{such\; that}\; \\ &\nonumber {for\; every}\;
p\in C\setminus B\;  {we\; have}\; u(p)\in W^s.\end{align}
\end{Def}

In other words, a quasimap is a map to the quotient stack $[W/\G]$ such that $u$ sends the generic point of
each irreducible component of $C$ to the stable locus $W^s$. The points in $B$ will be called the {\it base points}
of the quasimap.

\begin{Def}\label{stability}The quasimap $( C,p_1,\dots ,p_k, P, u)$ is called {\rm prestable} if the base points are disjoint from the nodes and markings on $C$.

It is called {\rm stable} if it is prestable and the line bundle
$\omega _C(\sum_{i=1}^kp_i) \ot {\mathcal L}^\epsilon $ is ample for every rational number $\epsilon > 0$, where
$${\mathcal L}:=\cL_\theta=u^*(P\times_\G L_\theta)\cong P\times_\G\CC_\theta 
.$$

\end{Def}

These definitions generalize and are motivated by two previous constructions in the literature.

\begin{EG} {\it Stable quotients:}
Consider, for $r \leq n$, the vector space $V = \mathrm{Hom}(\CC ^r, \CC ^n)=\mathrm{Mat}(n\times r)$ of $n\times r$ complex matrices,
equipped with the action of $\G= GL(r)$ via right multiplication and with stability induced by the determinant character.  In this case,
the associated GIT quotient is the Grassmannian $\mathrm{G}(r,n)$ and the definition of a stable quasimap
to $\VmodG$ simply recovers the definition of stable quotients first studied in \cite{MOP}.
It is shown there that quasimaps and stable maps to $\mathrm{G}(r,n)$ give essentially identical numerical invariants.
\end{EG}

\begin{EG} {\it Stable toric quasimaps:}
Suppose we are given a complete nonsingular fan $\Sigma \subset \RR^{n}$ generated
by $l$ one-dimensional integral rays inside an $n$-dimensional lattice.
There is a smooth $n$-dimensional projective toric variety
$X_\Sigma$ associated to the fan, which can be expressed as a GIT quotient of
$V = \CC^{l}$ by the torus $\G = (\CC^*)^{l-n}$.  In this case, we obtain stable toric quasimaps studied in \cite{CK}.
Already in this setting, the relation between stable quasimaps and stable maps is much more complicated, even at the numerical level.
\end{EG}

\begin{Def} An isomorphism between two quasimaps
$$( C,p_1,\dots ,p_k, P, u),$$
and
$$( C',p'_1,\dots ,p'_k, P', u'),$$
consists of an isomorphism $f: C \rightarrow C'$ of the underlying curves, along with an isomorphism
$\sigma: P \rightarrow f^*P'$, such that the markings and the section
 are preserved:
 $$f(p_j) = p'_{j}, \sigma_{W}(u) = f^*(u'),$$
 where $\sigma_W: P\times_\G W \rightarrow P'\times_\G W$ is the isomorphism of fiber bundles induced by $\sigma$.
 \end{Def}

\begin{Def}\label{family} A family of (stable) quasimaps to $\WmodG$ over a base scheme $S$ consists of the data
$$(\pi: \cC \rightarrow S, \{p_i: S \rightarrow \cC\}_{i=1,\dots,k}, \cP, u)$$
where
\begin{itemize}
\item $\pi: \cC \rightarrow S$ is a flat family of curves over $S$, that is a flat proper morphism of relative dimension one,
\item $p_i, i=1,\dots, k$ are sections of $\pi$,
\item $\cP$ is a principal $\G$-bundle on $\cC$,
\item $u:\cC\lra \cP\times_\G W$ is a section
\end{itemize}
such that the restriction to every geometric fiber $\cC_s$ of $\pi$ is a (stable) $k$-pointed quasimap of genus $g$ and class $\beta$.
An isomorphism between two such families $(\cC\rightarrow S,\dots)$ and $(\cC' \rightarrow S, \dots)$
consists of an isomorphism of $S$-schemes $f:\cC \rightarrow \cC'$, and an isomorphism of $\G$-bundles
$\sigma: \cP \rightarrow f^*\cP'$, which preserve the markings and the section.
\end{Def}

Stability of a quasimap $(C, p_1, \dots, p_k, P, u)$ implies immediately the following properties:
\begin{itemize}
\item Every rational component of the underlying curve $C$ has at least two nodal or marked points and $\deg(\cL)>0$ 
on any such component with exactly two special points; in particular, this forces the inequality $2g-2+k\geq 0$.

\item The automorphism group of a stable quasimap is finite and reduced.  Indeed, it suffices to restrict to any rational component $C'$
of $C$ with exactly two special points $q_1$ and $q_2$.  If the automorphism group $\CC^* = \mathrm{Aut}(C', q_1,q_2)$
preserves the section $u$, this forces $u$ to have no base points and $P|_{C'}$ to be trivial.  However, this violates stability.

\end{itemize}

%%%%%%%%%%%%%%%%%%%%%%%%%%%%%%%%%%%%%%%%%%%%%%%%%%%%%%%%%%%%%%%%
%%%%%%%%%%%%%%%%%%%%%%%%%%%%%%%%%%%%%%%%%%%%%%%%%%%%%%%%%%%%%%%%
\subsection{Boundedness}\label{bdd} We prove in this subsection that quasimaps with  fixed domain curve $C$ and class $\beta$, as well as
stable quasimaps with fixed numerical type $(g, k, \beta )$ form bounded families.

We start with the following important observation.

\begin{Lemma}\label{effective} If $(C,P,u)$ is a quasimap, then $\beta (L_\theta) \ge 0$; and $\beta (L_\theta) =0$ if and only if $\beta = 0$,
if and only if the quasimap is a constant map to $\WmodG$.
Furthermore, the same holds for any subcurve $C'$ and the induced quasimap.
\end{Lemma}
\begin{proof} It is enough to prove the the lemma when $C$ is irreducible.

Recall that the equivariant line bundle $L_\theta$ on $W$ descends to a 
line bundle $\cO(\theta)$ on $\WmodG$ which is relatively ample over the affine quotient.
Pick $m> 0$ such that $\cO(m\theta)$ is relatively very ample and let 
\begin{equation}\label{projemb}\WmodG\hookrightarrow \PP_{A(W)^\G}((\Gamma(W,L_{m\theta})^\vee)^\G)\end{equation}
be the corresponding embedding over $\Spec(A(W)^\G)$.
If $u$ has image contained in $P\times_\G W^s$, then it corresponds 
to a regular map $f:C\lra\WmodG$ and $\beta (L_\theta) =\deg_C f^*\cO(\theta)$. The claim is clear in this case. Assume now that there is a base-point
$x\in C$ (so that $u(x)$ is unstable). Via composition with the inclusion (\ref{projemb}), we view the quasimap as
a {\it rational} map 
$$C- - \ra \PP_{A(W)^\G}((\Gamma(W,L_{m\theta})^\vee)^\G)$$
with a base-point at $x$.
It follows that every $\G$-equivariant section $t$ of $L_{m\theta}$ must satisfy $u^*t(x)=0$.
On the other hand, since the image of the generic point is stable, the line bundle $\cL_{m\theta} =u^*(P\times_\G L_{m\theta})$
on $C$ has a nonzero section of the form $u^*t$. Hence $m\beta(L_\theta)=\deg(\cL_{m\theta})>0$.
\end{proof}

\begin{Def}\label{eff-semigroup} We will call elements $$\beta\in\Hom_\ZZ(\Pic^\G(W),\ZZ)$$
which are realized as classes of quasimaps (possibly with disconnected domain curve) to $W/\!\!/_\theta\G$
($\G$-equivariant) {\it $L_\theta$-effective classes}. These $L_\theta$-effective classes form a semigroup, 
denoted $\mathrm {Eff}(W,\G,\theta)$. 
\end{Def}
Hence, if $\beta\in \mathrm {Eff}(W,\G,\theta)$, then either $\beta=0$ or $\beta (L_\theta ) >0$.

An immediate consequence of Lemma \ref{effective} is a boundedness result for the underlying curves of stable
quasimaps:

\begin{Cor}\label{boundedness-curve}
The number of irreducible components of the underlying curve of a {\em stable} quasimap to $\VmodG$ is bounded
in terms of $g,k$, and $\beta$ only.
\end{Cor}
\begin{proof} See Corollary 3.1.5 in \cite{CK}.
\end{proof}

Next we bound quasimaps with given class $\beta$ on a nodal curve. 

\begin{Thm}\label{boundedness-qmap} Let $\beta\in\Hom_\ZZ(\Pic^\G(W),\ZZ)$  
and a nodal curve $C$ be fixed. The
family of quasimaps of class $\beta$ from
$C$ to $\WmodG$ is bounded.
\end{Thm}

We begin by making some reductions.

First, by Proposition \ref{embedding}, it suffices to assume that $W=V$ is a vector space. Indeed, given a 
$\G$-equivariant embedding $W\subset V$, 
every quasimap $(C,P,u)$
to $\WmodG$ is also a quasimap to $\VmodG$, where we use the same linearization $\theta\in\chi(\G)$ 
to define the stable and semistable loci on $V$. Note that the class of the quasimap to $\VmodG$ is the image of $\beta$
under the natural map 
$$\Hom(\Pic^{\G}(W),\ZZ)\lra\Hom(\Pic^{\G}(V),\ZZ)=\Hom(\chi(\G),\ZZ), $$
and is in fact just the degree of the underlying principal $\G$-bundle $P$.
Strictly speaking, we are in a slightly more general situation since
the linearized action on $V$ may fail to satisfy the assumptions in $\S 2.2$: there may be strictly semistable points in $V$,
and points in $V^s\setminus W^s$ may have nontrivial stabilizers. However, the proof will only use
that $V^s\neq\emptyset$, and that the representation $\G\lra GL(V)$ has finite kernel. These follow respectively from
the fact that $W^s$ is nonempty, and from the freeness of the $\G$-action on $W^s$ (which implies the kernel is trivial).

Second, notice that the
quasimaps $(C,P,u)$ to $\VmodG$ with fixed projective, connected nodal curve $C$ and fixed principal $\G$-bundle $P$ on it,  are 
parametrized by an open subset of the space of global sections of the induced {\it vector bundle}
$$\cV_P=P\times_\G V$$
on C. Hence, after fixing $C$ and a degree $$\beta\in \Hom(\Pic^\G(V),\ZZ)=\Hom(\chi(\G),\ZZ),$$ 
it suffices to bound the set $S$ of principal $\G$-bundles $P$ on $C$ 
admitting a quasimap $(P, u)$ of class $\beta$ 
to $\VmodG$.

Finally, a principal $\G$-bundle on a nodal curve $C$ is given by a principal bundle $\widetilde{P}$ on the
normalization $\widetilde{C}$, together with identifications of the fibers $\widetilde{P}_x$ and $\widetilde{P}_y$
for each pair of preimages of a node. For each node, these identifications are parametrized by the group $\G$. It
follows that we may assume that the curve $C$ is irreducible and nonsingular.

Hence Theorem \ref{boundedness-qmap} is a consequence of the following result:

\begin{Thm}\label{boundedness-bundle} Let $\beta\in\Hom_\ZZ(\chi(\G),\ZZ)$  
and a smooth projective curve $C$ be fixed. Let $V$ be a vector space with an action of $\G$
via a representation $\G\lra GL(V)$ with finite kernel and let $\theta\in\chi(\G)$ be a character such that $V^s(\G,\theta)\neq\emptyset$.
Then the
family of principal $\G$-bundles $P$ on $C$ of degree $\beta$
such that the vector bundle $\cV_P=P\times_\G V$ admits a section $u$ which sends the generic 
point of $C$ to $V^s(\G,\theta)$ is bounded.
\end{Thm}

\begin{proof} We first prove the Theorem under the additional assumption that $\G$ is {\it connected}, using an argument similar to the one employed by
Behrend  \cite{Behrend} to show boundedness for semistable principal $\G$-bundles
of fixed degree on a curve. Essentially the same proof also appeared later in the paper \cite{Hol-Nar}
of Holla and Narasimhan.
It is based on two general auxiliary results which we state below.

Let $\bT$ be a maximal torus in $\G$ and let $\bB$ be a Borel subgroup of $\G$ with $\bT\subset \bB$. Let
$\bB_{r}$ be the unipotent radical of $\bB$. Note that the composition $\bT\subset \bB\rightarrow \bB/\bB_{r}$ provides an isomorophism $\bT\cong \bB/\bB_r$;
and $\chi (\bT)=\chi (\bB)$ since the unipotent radical has only the trivial character.

\begin{Lemma}\label{Breduction} Every principal $\G$-bundle on a smooth projective curve 
admits a reduction to a principal $\bB$-bundle.  
\end{Lemma}
\begin{proof} This is a corollary of
Springer's result that any principal $\G$-bundle on a projective smooth curve $C$ is locally
trivial in the Zariski topology of $C$ (see \cite[\S 2.11]{Ram}).
\end{proof}

\begin{Lemma}\label{Tbounded} Let $S$ be a set of principal $\G$-bundles on $C$ with
chosen $\bB$-reductions $P'$ of $P$ for each $P\in S$ and consider
the set of 
associated $\bT=\bB/\bB_r$-bundles
\begin{equation}\label{R} R := \{ \overline{P'}=P'/\bB_r : P\in S\}. \end{equation}
If $R$ is bounded, then $S$ is also bounded. Furthermore, $R$ is bounded if the
set of degrees
\begin{equation}\label{degrees}\{d_{\overline{P'}}:\chi(\bT)\lra\ZZ : \overline{P'}\in R  \}\end{equation}
is a finite set. (Recall that $d_{\overline{P'}}$ is defined by $d_{\overline{P'}}(\xi)=\deg(\overline{P'}\times_\bT\CC_\xi)$.)
\end{Lemma}

\begin{proof} See \cite[Proposition 3.1]{Hol-Nar} and \cite[Lemma 3.3]{Hol-Nar}.
\end{proof}

Now let $S$ be the set of principal $\G$-bundles in Theorem \ref{boundedness-bundle} and for
each $P\in S$ pick a $\bB$-reduction $P'$. This is possible by Lemma \ref{Breduction}. The following  Lemma is the
main step in the proof.

\begin{Lemma}\label{BddLemma} There is a $\mathbb{Q}$-basis $\{\theta_i\}$ of $\chi (\bT)\ot \QQ$
such that: \begin{enumerate}

\item For each $i$, and each $P'$, the line bundle 
$$P'\times _\bB \mathbb{C} _{\theta _i} = P'/\bB_r \times _\bT \mathbb{C}_{\theta_i}$$ 
on $C$ has nonnegative degree.

\item  $\theta = \sum a_i \theta_i$ for some nonnegative rational numbers $a_i$,
where $\theta$ is considered as a character of $\bT$.
\end{enumerate}
\end{Lemma}

\begin{proof}
There is an open neighborhood $A$ of $\theta$ in $\chi (\bT)\otimes \mathbb{Q}$ such that 
\begin{equation}\label{inc} V^s(\bT, \theta ) \subset V^{ss}(\bT, \xi ) \end{equation} for every $\xi \in A$.
Here $V^{s}(\bT, \theta )$ (resp.  $V^{ss}(\bT, \xi )$)
denotes the $\theta$-stable (resp. $\xi$-semistable) locus of $V$ for the action of $\bT$ linearized by the given character; and
a point $p\in V$ is said to be $\xi$-semistable if it is so for some
positive integer $\ell$ making $\ell\xi$ an integral character in $\chi (\bT)$.
The existence of $A$ can be seen by Mumford's Numerical Criterion for (semi)stability, in the form Proposition 2.5 in \cite{King}, as follows.

Decompose $V$ as a direct sum of one-dimensional $\bT$-eigenspaces with 
(not necessarily distinct) weights $\delta _j$. 
By the Numerical Criterion, a point $p=(p_1,\dots ,p_{\dim V})\in V$ is $\theta$-stable
if and only if $(a_1p_1,\dots ,a_{\dim V}p_{\dim V})$ is $\theta$-stable for every $a_i\in \mathbb{C} ^*$.
Hence it is enough to check the inclusion \eqref{inc} only at the points $p$ with coordinates $p_j=0$ or $1$. 
For such a point $p$ which is $\theta$-stable,
consider the cone $\sigma (p)$ in  $(\chi (\bT)\otimes \mathbb{Q})^\vee$
consisting of $\lambda$ such that $\langle \delta _j, \lambda \rangle \ge 0$ 
for every $j$ with $p_j=1$. If we let $A$ be the intersection of the (finitely many) dual open cones
$\sigma (p) ^\vee$, then \eqref{inc} follows immediately from the Numerical Criterion.

Via the homomorphism $\G\lra GL(V)$, the image of the Borel subgroup $\bB$ is
contained in a Borel subgroup of $GL(V)$. 
It follows that we have a decomposition
\begin{equation}\label{dec} V = \oplus _{l=1}^{\dim V} V_l/V_{l-1} \end{equation}
 such that $V_l$ is a $\bB$-invariant $l$-dimensional subspace
of $V$. 
The two sides of \eqref{dec} are isomorphic as $\bT$-representation spaces.
Hence from the $\bB$-equivariant map 
\[ \begin{array}{ccccc} P' &\rightarrow & P'\times _\bB\G = P &\rightarrow & V \\
                                   p' &\mapsto &(p',e) & \mapsto & u(p',e) \end{array}\]
given by a quasimap $u$ we obtain a $\bT$-equivariant map 
\begin{equation}\label{Tequiv}P'/\bB_r \rightarrow \oplus _{l=1}^{\dim V} V_l/V_{l-1}\end{equation}
Via composition with \eqref{Tequiv}, any
$\bT$-equivariant map $$s: \oplus _{l=1}^{\dim V} V_l/V_{l-1}=V \rightarrow \mathbb{C}_\xi$$
induces a section $[s]$ of $P'/\bB_r \times _\bT \mathbb{C}_\xi$. 
Since $V^{s}(\G, \theta)\subset V^s(\bT, \theta ) \subset V^{ss}(\bT, \xi)$ for $\xi \in A$, and $u$ is a quasimap,
this section is nonzero at the generic point of $C$. Hence $\deg(P'/\bB_r \times _\bT \mathbb{C}_\xi)\geq 0$ for every
$\xi\in A$.
It is clear that there is a basis $\{\theta_i\}$ of $\chi(\bT)\ot\QQ$ contained in $A$ and satisfying condition $(2)$ of the Lemma.
\end{proof}

The proof of Theorem \ref{boundedness-bundle} for connected $\G$ is now immediate.
First, $\beta(\theta)\geq 0$ by Lemma \ref{effective}.
Next, by Lemma \ref{BddLemma} 
$$\beta(\theta)=\deg (P\times _\G\mathbb{C}_\theta) =\deg (P'\times _\bB\mathbb{C}_{\sum a_i\theta _i} )= \sum a_i \deg  (P'\times _\bB\mathbb{C}_{\theta _i}),$$
for every $P$.
Since all $a_i$'s are nonnegative, we see that 
$$0\le \deg (P'\times _\bB \mathbb{C} _{\theta _i})=\deg( \overline{P'} \times _\bT \mathbb{C} _{\theta _i} )$$ is uniformly bounded above
as $P$ varies.  Since the $\theta_i$'s form a basis
of $\chi(\bT)\ot\QQ$, it follows that the set of degrees in \eqref{degrees} is finite, so we conclude by Lemma \ref{Tbounded}.

We now consider the case of a general reductive group $\G$. Let $\G_\circ$ be the connected component of the identity $e\in \G$. It is
a normal subgroup and $\G/\G_\circ$ is a finite group. Given a principal $\G$-bundle $P$ on $C$ we have a factorization
$$\begin{array}{ccc}P&\stackrel\pi\lra &P/\G_\circ \\ 
\;\;\;\rho\searrow& &\swarrow q\;\;\;\;\;\;\\ 
&C &
\end{array}
$$
with $P\lra P/\G_\circ$ a principal $\G_\circ$-bundle and $P/\G_\circ\stackrel q\lra C$ a principal $\G/\G_\circ$-bundle, so that
$\widetilde{C}:=P/\G_\circ$ is a nonsingular projective curve and $q$ is a finite {\' e}tale map of degree $m:=|\G/\G_\circ|$.
Note that there are only finitely many possibilities for $\widetilde{C}$ when $P$ varies, since there are only finitely many
principal $\G/\G_\circ$-bundles on $C$.

Let $\widetilde{\beta}:\chi(\G_\circ)\ra\ZZ$ be the degree of the $\G_\circ$-principal bundle $P\lra\widetilde{C}$.
Then $\beta(\theta)=m\widetilde{\beta}(\theta)$, where we view $\theta$ as a character of $\G_\circ$ by restriction, so
$\widetilde{\beta}(\theta)$ is fixed, independent on the principal $\G$-bundle $P$ of degree $\beta$. 

Finally, any $\G$-equivariant map $P\lra V$ is also $\G_\circ$-equivariant; furthermore 
$$V^{ss}(\G,\theta)=V^{ss}(\G_\circ,\theta)\; {\mathrm {and}}\; V^{s}(\G,\theta)=V^{s}(\G_\circ,\theta)$$
by \cite[Prop 1.15]{Mumford}. Hence if $P\times_\G V\lra C$ admits a section $u$ as in the Theorem, so does
$P\times_{\G_\circ} V\lra \widetilde{C}$. Noting that the proof in the case of connected group only used that
$\beta(\theta)$ was fixed, we conclude from that case that the sub-family of $P$'s with fixed $\widetilde{C}$
is bounded, so we are done by virtue of the finiteness of the set of $\widetilde{C}$'s.
\end{proof}

For later use in section \ref{vspace} we record the following immediate consequence of Theorem \ref{boundedness-bundle}.
\begin{Cor}\label{regularity} Let $$(\cC\lra S, p_1,\dots, p_k: S\lra\cC, \cP, u)$$ be a family
of quasimaps of genus $g$ and class $\beta$ to $\VmodG$ with $S$ a scheme of finite type and let
$\cO_\cC(1)$ be an $S$-relatively ample line bundle on $\cC$.
There exists an integer $m>\!\!> 0$, depending on $k,g$, and $\beta$, but \underline{independent} on $\cP$ and $u$, such that the restriction
to every geometric fiber of the vector bundle $\cV_\cP(m)=(\cP\times_\G V)\ot\cO_\cC(m)$
on $\cC$ has vanishing higher cohomology.
\end{Cor}
\begin{Rmk} As we already observed, the proof of Theorem \ref{boundedness-bundle} only required that $\beta(L_\theta)$ is fixed. 
The same is true about the proof of Corollary \ref{boundedness-curve}. It follows
from the results of this subsection that for fixed $d\in\ZZ_{\geq 0}$ there are only finitely many $\beta\in \mathrm {Eff}(W,\G,\theta)$
with $\beta(L_\theta)=d$.
\end{Rmk}

%%%%%%%%%%%%%%%%%%%%%%%%%%%%%%%%
%%%%%%%%%%%%%%%%%%%%%%%%%%%%%%%%%%%%%%%%%%%%%%%%%%%%%%%%%%%%%%%%
\section{Moduli of stable quasimaps to $\WmodG$}

\subsection{The moduli stack}
\begin{Def} Let $k, g$ be fixed with $2g-2+k\geq 0$, and let $\beta$ be a $L_\theta$-effective class.
The moduli stack $$\mathrm{Qmap}_{g,k}(\WmodG,\beta)$$ of $k$-pointed quasimaps of genus $g$ and class
$\beta$ to $\WmodG$ is the stack parametrizing families as in Definition \ref{family}

\end{Def}

Strictly speaking, we only have a prestack, but it is straightforward to check the sheaf condition. 

There are natural forgetful morphisms of stacks
$$\QmapW\stackrel\mu\lra \BunG\;\;\;\mathrm{and}\;\;\; \QmapW\stackrel\nu\lra\fMgk.$$

\begin{Thm}\label{Thm1} 
The moduli stack $\mathrm{Qmap}_{g,k}(\WmodG,\beta)$ is a Deligne-Mumford stack of finite type,
with a canonical relative obstruction theory over $\BunG$. If $W$ has only lci singularities, then the obstruction theory is perfect.
If $\WmodG$ is projective, then $\QmapW$ is proper over $\Spec(\CC)$. In general, it has
a natural proper morphism
$$\QmapW\lra \mathrm{Spec}(A(W)^\G).$$

\end{Thm}

We will split the proof of the Theorem over several subsections below.

\subsection {The moduli stack is Deligne-Mumford and of finite type}\label{construction}
Consider the forgetful morphisms 
$$\QmapW\stackrel\mu\lra \BunG\;\;\;\mathrm{and}\;\;\; \QmapW\stackrel\nu\lra\fMgk.$$

By a standard argument it follows from Corollary \ref{boundedness-curve} that the map $\nu$ factors through an open substack
of finite type $\fS\subset\fMgk$. Further, by Theorem \ref{boundedness-bundle}
there is an open substack of finite type
$\BunGbeta\subset\BunG$, fitting in
a commutative diagram
$$\begin{array}{ccc}
\BunGbeta&\subset &\BunG \\
\downarrow & &\downarrow \\
\fS &\subset &\fMgk
\end{array}$$
and such that $\mu$ factors through $\BunGbeta$. This is again standard for $\G=GL_n$: 
the stack $\fS\times_{\fMgk}\BunG$ has a smooth surjective cover $\mathcal{U}$ which is an infinite disjoint union
of quot schemes and Theorem \ref{boundedness-bundle} implies that $\mu$ factors through the image of only finitely
many such quot schemes; we take $\BunGbeta$ to be this image. The case general of a general group $\G$ is 
reduced to $GL_n$ as in the proof of Proposition \ref{bunG}. 

\phantom{X}

{\it Claim}: The morphism $\mu$ is representable and of finite type. 
More precisely, its fibers are open subschemes
in appropriate Hilbert schemes.

\phantom{X}

Granting the claim, it follows that $\QmapW$ is an Artin stack of finite type.
By stability, the automorphism group of a stable quasimap is finite and reduced, so we get in fact a DM stack.

{\it Proof of the Claim.}  Let $\fC\stackrel\pi\lra\BunGbeta$ be the universal curve, with universal
sections $\mathfrak{p}_i,\; i=1,\dots,k$, and let $\fP$ be the universal
principal $\G$-bundle over $\fC$. If $\omega_\fC$ denotes the relative dualizing sheaf, then we have 
line bundles 
$$\mathfrak{L}=\mathfrak{L}_\theta=\fP\times_\G\CC_\theta \;\; \mathrm{and}\;\; \omega_\fC(\sum\mathfrak{p}_i)\otimes\mathfrak{L}^{\ot \epsilon},\; \epsilon\in\QQ$$ 
on $\fC$.
There is an open substack 
$\BunGbeta^\circ\subset\BunGbeta$ obtained by requiring that $\omega_\fC(\sum\mathfrak{p}_i)\otimes\mathfrak{L}^{\ot \epsilon}$ is 
$\pi$-relatively ample for some positive $\epsilon$. 
We will continue to denote  by $\fC\stackrel\pi\lra\BunGbeta^\circ$ and $\fP$ the universal curve and universal principal
$\G$ bundle restricted to this open substack. Note that $\pi$ is now a {\it projective} morphism.

Forming the induced bundle $\fP\times_\G W$ with fiber $W$, we have a
commutative diagram of projections
$$\begin{array}{ccc}\fP\times_\G W&\stackrel\varrho\lra &\fC \\ \\
\;\;\;\;\;q\searrow& &\swarrow\pi \;\;\;\;\;\;\\ \\
&\BunGbeta^\circ &
\end{array}
$$
The fiber of $q$ over a $\CC$-point $(C,p_1,\dots, p_k, P)$ of $\BunGbeta^\circ$ is $P\times_\G W$.

Let $\overline{W}$ be a projective $\G$-equivariant completion of $W$. For each
scheme $T$ with a morphism
$$T\lra\BunGbeta^\circ,$$
consider the pulled-back universal curve $\cC_T\stackrel\pi\lra T$ and universal bundle $P_T$ on $\cC_T$.
A section $u$ of
$$P_T\times_\G \overline{W}\stackrel\varrho\lra\cC_T$$
can be identified with its image inside $P_T\times_\G \overline{W}$.
It is easily checked that the Hilbert polynomial of the image is determined by $g$ and $\beta$,
hence $u$
corresponds to a point in the
(relative) Hilbert scheme
$$\mathrm{Hilb}_{\beta,g}(P_T\times_\G \overline{W}/T)\lra T.$$

In the terminology of \cite[\S 14]{LM}, associating this Hilbert scheme to each $T\lra\BunGbeta$ gives a {\it schematic
local construction of finite type}, so we have an Artin stack with a schematic morphism of finite type
$$\mathfrak{Hilb}_{\beta, g}(\fP\times_\G\overline{W}/\BunGbeta^\circ)\stackrel\mu\lra\BunGbeta^\circ,$$
whose points parametrize closed subschemes $Y$ (with Hilbert polynomials specified as above) in the fibers
of the projection
$$q: \fP\times_\G\overline{W}\lra\BunGbeta^\circ.$$
We now define a substack $\cH$ of $\mathfrak{Hilb}_{\beta, g}(\fP\times_\G\overline{W}/\BunGbeta^\circ)$ by successively imposing the following requirements:
\begin{enumerate}
\item the subscheme $Y$ lies inside $P\times_\G W$
\item $Y$ is the image of a section $u:C\lra P\times_\G W$
\item the section $u$ maps the generic points of components of $C$, the nodes, and the markings into $P\times_\G W^s$
\item stability holds, i.e., $\omega_C(\sum p_i)\otimes\cL^{\ot\epsilon}$ is ample for {\it every} $\epsilon >0$.
\end{enumerate}

Since each of these is an open condition, $\cH$ is an open substack, which is obviously identified with $\QmapW$. The claim is proved.

\begin{Rmk}\label{general} Let $\fMgk([W/\G],\beta)$ denote the moduli stack 
parametrizing maps of class $\beta$ from prestable $k$-pointed curve of genus $g$
to the quotient stack $[W/\G]$. Using the representability of the Hilbert functor in algebraic spaces,
the argument above shows that the natural morphism
$$\mu:\fMgk([W/\G],\beta)\lra\BunG$$ 
is representable and of finite type.
Hence the stack $\fMgk([W/\G],\beta)$ is algebraic and locally of finite type. (As for $\fBun$, this conclusion
also follows immediately from \cite[Prop. 2.18]{Lieblich} and \cite[Lemma C.5]{AOV}.)
Theorem \ref{boundedness-qmap} then shows that imposing the quasimap condition (Definition \ref{qmap})
determines an open substack of $\fMgk([W/\G],\beta)$ which is of finite type over $\fMgk$. 
Finally, imposing in addition a stability condition on quasimaps, such as the one in Definition \ref{stability},
gives an open substack which is Deligne-Mumford and of finite type.

\end{Rmk}

%%%%%%%%%%%%%%%%%%%%%%%%%%%%%%%%%%%%%%%%%%%%%%%%%%%%%%%%
\subsection{Properness}

\begin{Prop}\label{propernessProp} 
The stack $\QmapW$ is proper over the affine quotient $\Spec(A(W)^\G)$.
\end{Prop}

\begin{proof} Assume first for simplicity that $A(W)^\G=\CC$, so that $\WmodG$ is projective.
The proof that $\QmapW$ is separated follows from the valuative criterion by 
an argument identical to the one for the toric case, given in section $4.1$ of \cite{CK}
(in turn, that argument is a straightforward modification of the one for stable quotients in \cite{MOP}).
For the more general case here, one needs the statement that if two principal $\G$-bundles on a regular
complex surface $S$ agree outside finitely many points, then they agree on $S$. This follows from Hartogs' theorem.
We should also point out that the ``prestable"
assumption that the base-points of stable quasimaps are away from nodes and markings is crucially used
in the proof.

It remains to show completeness. For this we use the valuative criterion and standard semistable reduction
techniques.
Let $(\Delta, 0)$ be a smooth curve with a point $0$ and let $\Delta ^\circ = \Delta \setminus\{0\}$. 
Let $((C, p_i), P, u)$ be a $\Delta ^\circ$-family of stable quasimaps. It is enough to show that, possibly after replacing $(\Delta, 0)$ 
by making a base-change ramified over $0$, there
is a stable quasimap extension of $((C, p_i), P, u)$ over $\Delta$. Let $B$ denote the base locus of this family with reduced scheme
structure. 
As in \cite{Ful-Pan, MOP, CK}, after possibly shrinking $\Delta$ and making an \' etale base change, we may regard $B$ as additional sections $y_j$.
Then $(C,p_i,y_j)\lra\Delta^\circ$ is a family of prestable curves.
Further, we may assume that $C\lra\Delta^\circ$ has smooth irreducible fibers and that the total space is a smooth surface.

The given family of quasimaps $((C, p_i), P, u)$ gives a 
rational map $$[u]:C - - \ra W/\!\!/G,$$ 
which is regular on $C\setminus B$. Since $\WmodG$ is projective, after shrinking $\Delta$ if necessary, it extends to 
a regular map, denoted by $[u_{reg}]$, on all of $C$. Stability of the quasimaps implies immediately
that $((C,p_i,y_j),[u_{reg}])$ is a $\Delta^\circ$-family of {\it stable maps} to $\WmodG$. We may view it also as a 
family of quasimaps, and as such it has a class $\beta_{reg}\in \Hom(\Pic^\G(W),\ZZ)$; the image of $\beta_{reg}$
in $\Hom(\Pic(\WmodG),\ZZ)$ is the element determined by the usual homology class of the stable map. In general $\beta$ and $\beta_{reg}$
are different. In fact, by Lemma \ref{degree-splitting} from \S 7 below, $\beta -\beta_{reg}$ is $L_\theta$-effective and 
it is nonzero if the base locus is nonempty.

By the properness of the
stable map functor, there is a family of stable maps 
$$(\widehat{C},p_i,y_j)\lra\Delta,\;\;\;\; [\widehat{u}]:\widehat{C}\lra\WmodG$$
extending $((C,p_i,y_j),[u_{reg}])$. The surface $\widehat{C}$ is normal and may have singularities only at
nodes in the central fiber.

Pulling back via $[\widehat{u}]$ the principal $\G$-bundle $W^s\lra\WmodG$,  we obtain a pair
\begin{equation}\label{Phat}(\widehat{P},\widehat{u})\end{equation}
 consisting of a principal $\G$-bundle $\widehat{P}$ on $\widehat{C}$ and an induced section $\widehat{u}:\widehat{C}\lra \widehat{P}\times_\G W$.

The central fiber $\widehat{C}_0$ of $\widehat{C}$ may contain ``rational tails" with respect to the original set of
markings. Specifically, consider all maximal trees $\Gamma_1,\dots,\Gamma_N$ of rational curves in $\widehat{C}_0$ such that each $\Gamma_l$
contains none of the markings $p_i$ and meets the rest of the curve $\overline{(\widehat{C}_0\setminus\Gamma_l)}$ in a single 
point $z_l$. 
Each of these trees can be contracted by a succession of blow-downs of $(-1)$-curves.

Denote by $\bar{C}$ the 
surface obtained by contracting the trees $\Gamma_1,\dots,\Gamma_N$. It is a $\Delta$-family of prestable curves with markings $p_i$,
while some of the additional sections $y_j$ may come together at points $z_l$ in the central fiber where contracted trees were attached.
Note that each of $z_1,\dots, z_N$ is a nonsingular point of $\bar{C}$. 
The pair $(\widehat{P},\widehat{u})$ from \eqref{Phat} restricts to a pair $(P_1,u_1)$ 
on the open subset $$U_1:=\bar{C}\setminus ((\cup_j y_j)\cup\{z_1,\dots,z_N\}).$$
On the other hand, we have the original pair $(P,u)$ on the open subset
$$U_2:=\bar{C}\setminus\bar{C}_0\cong C.$$
By construction, $(P_1,u_1)$ and $(P,u)$ agree on the intersection $U_1\cap U_2$, hence they glue
together to give a pair $(\bar{P}^0,\bar{u}^0)$ defined on all of $\bar{C}$, except at
the points $z_l$ where rational tails were contracted, and at the intersections of the sections $y_j$ with the central fiber.
These points are all away from the nodes and markings in the central fiber.

Now, Lemma \ref{Pext} below gives an extension $\bar{P}$ of $\bar{P}^0$ to $\bar{C}$.  
Furthermore, $\bar{u}^0$ is a section of $\bar{P}\times _G W$ defined over $\bar{C}$  except at finitely many nonsingular points in the
central fiber. By Hartogs' theorem, we can extend the section as well to a section $\bar{u}$ defined on all of $\bar{C}$.
This is the key place where $W$ affine is important for the argument. 

We have constructed a family of prestable quasimaps
$$((\bar{C},p_i),\bar{P},\bar{u}).$$
The limit quasimap will have base-points precisely at the points $z_l$ where rational tails were contracted, 
and at the intersections of the sections $y_j$ with the central fiber. This is because the rational tails
in the stable map limit 
$((\widehat{C}_0,p_i,y_j), [\widehat{u}]|_{\widehat{C}_0})$
all carried either a nonzero part of $\beta_{reg}$, or some of the points $y_j$ (or both), 
while each of the points $y_j$ correspond to a nonzero part of the difference $\beta-\beta_{reg} $.
(For a more precise analysis of the degrees carried by the base points, see the proof of Theorem \ref{Thm3} in \S 7.)
It is now immediate to see that the limit is in fact stable. Indeed, we need to check that the line bundle
$\cL_\theta$ is nontrivial on rational components with only two special points (nodes or markings $p_i$)
in the central fiber. This is clear if the quasimap has a base-point on the component.
However, if such a component doesn't contain any of the points $y_j$ and $z_l$, then by construction the 
restriction of the quasimap to it is a stable map to $\WmodG$, hence non-constant, and $\cL_\theta$ is the pullback of the 
polarization on $\WmodG$. 
From the valuative criterion we conclude that $\QmapW$ is proper.

Assume now that we are in the general situation when we have a projective morphism
$$\WmodG\lra\Spec(A(W)^\G).$$ 
The morphism from $\QmapW$ to the affine quotient is obtained as follows:
First of all, it is well-known that
$$q:W\lra\mathrm{Spec}(A(W)^\G)$$
is a categorical quotient for the $\G$-action on $W$. 
Therefore, for every quasimap
$$(C,P,\tilde{u}:P\lra W)$$
the $\G$-equivariant morphism $\tilde{u}$ induces a map
$$C\lra\mathrm{Spec}(A(W)^\G),$$
which must be constant, since $C$ is projective and $\mathrm{Spec}(A(W)^\G)$ is affine.

More generally, the universal family
$$(\cC\lra\QmapW, \fP,\tilde{u}:\fP\lra W)
$$
induces a map from the universal curve $\cC$ to the affine quotient. By the previous paragraph,
this map is constant on the fibers of the projection $\pi:\cC\lra\QmapW$.
Hence it descends to our desired morphism
$$\eta:\QmapW\lra \mathrm{Spec}(A(W)^\G).$$

The properness of $\eta$ now follows from the valuative criterion, by the relative version of argument given above in the
absolute case.
 \end{proof}
 
\begin{Lemma}\label{Pext}
 Let $S$ be  an algebraic surface (irreducible and reduced) and let $P^0$ be a principal $\G$-bundle on 
 $S$ minus a finite set of nonsingular points. 
 Then $P^0$ extends to a principal $\G$-bundle on $S$.
 \end{Lemma}

 \begin{proof}
Replacing $S$ by its nonsingular locus if necessary, we may assume that it is nonsingular.
Let $U$ denote the complement of a finite set over which $P^0$ is defined.

In the case $\G = GL(n)$, we can replace $P^0$ with a rank $n$ vector bundle $E^0$ over $U$.
It suffices to show that this extends to $S$.  There exists a rank $n$ coherent sheaf $E$ on $S$ extending $E^0$;
by replacing $E$ with its double-dual, we can assume $E$ is reflexive.
Since $S$ is two-dimensional and regular, this forces $E$ to be locally free, giving the extension of $P^0$.

In the case of general $\G$, we choose some embedding $\G \subset GL(n)$.
Let $Q^0 = P^0\times_{\G} GL(n)$ be the associated $GL(n)$-bundle over $U$.
By the last paragraph, $Q^0$ extends to a $GL(n)$-bundle $Q$ over all of $S$.
 Reductions of $Q$ to a principal $\G$-bundle correspond to sections of
 the associated $GL(n)/\G$-bundle $Q/\G$ over $S$.  Since $\G$ is reductive,
 it follows from Matsushima's criterion \cite{matsushima} that the homogeneous space
$GL(n)/\G$ is affine.  Moreover, by assumption, $Q/\G$ admits a section over $U$.
Thus, by Hartogs' theorem, this section extends over $S$.
 \end{proof}

\begin{Rmk}
Note that the statement of the lemma is false if $\G$ is not reductive: for any linear algebraic group with nontrivial unipotent radical, there exist principal $\G$-bundles over $U = \CC^2\setminus \{0\}$ which do not extend to the entire plane. Indeed, for $\G = \G_a$, principal bundles over $U$ are classified by $H^1(U, \mathcal{O}_U)\ne 0$.  Since any unipotent $\G$ is built from $\G_a$ by taking successive extensions, the same calculation yields the statement for unipotent groups.  Finally if $\G$ has nontrivial unipotent radical $\mathbf{U}$, then the long exact sequence for group cohomology shows that the natural map $H^{1}(U,\mathbf{U}) \rightarrow H^{1}(U,\G)$ is injective and any elements in its image cannot be extended to the plane. 
\end{Rmk}

\subsection{Obstruction theory} Let
$$\cC \stackrel\pi\lra\QmapW$$ be the universal curve, with universal principal $\G$-bundle $\mathfrak{P}$ on it.
($\mathfrak{P}$ is the pull-back of the universal bundle on the universal curve over $\BunG$.) Let
$$\varrho : \mathfrak{P}\times_\G W\lra\cC$$ be the induced bundle with fiber $W$ and let
$$u:\cC\lra \mathfrak{P}\times_\G W$$ be the universal section.

Recall from the construction in $\S\ref{construction}$ that the fibers of the map 
$$\mu:\QmapW\lra\BunG$$ are open in Hilbert schemes, with points corresponding to graphs of the sections $u$.
It follows that there is a {\it canonical} relative obstruction theory for $\QmapW$ over $\BunG$ given by
\begin{equation}\label{relobstruction}
E_\mu^\bullet:=\left (R^\bullet\pi_*({\cH}om(\mathbb{L}_u,\cO_\cC)[1])\right )^\vee
\end{equation}
where $\mathbb{L}_u$ is the relative cotangent complex of $u$.

\begin{Prop}\label{smooth} If $W$ is smooth, then the relative obstruction theory over $\BunG$ is perfect.
\end{Prop}

\begin{proof} When $W$ is smooth the morphism $u$ is an lci closed embedding, hence its cotangent complex is
quasi-isomprphic to a single locally
free sheaf in degree $-1$ (the conormal bundle of the section). Since $\pi$ is projective of relative dimension 1, the proposition follows.
\end{proof}

\subsection{Perfect obstruction theory when $W$ has lci singularities}

Note that since $u$ is a section of $\varrho$, we have
$$\mathbb{L}_u\cong u^*\mathbb{L}_\varrho[1],$$ hence we can rewrite the relative obstruction theory as
$$\left (R^\bullet\pi_*(u^*\mathbb{R}T_\varrho)\right )^\vee,$$
with $\mathbb{R}T_\varrho$ the relative tangent complex of $\varrho$.

Assume now that $W$ has lci singularities only (contained in $W\setminus W^s$).
Choose a nonsingular $\G$-variety $V$ and
a $\G$-equivariant embedding $W\hookrightarrow V$ corresponding to an ideal
sheaf $I\subset \cO_V$. For example, we can take $V$ to be a vector space, as in 
Proposition \ref{embedding}.
The tangent complex of $W$ is then the two-term complex of vector bundles
$$\mathbb{R}T_W=[T_V|_W\lra (I/I^2)^\vee]$$
up to quasi-isomorphism.

We have an induced complex of vector bundles
$$[\mathfrak{P}\times_\G T_V|_W\lra \mathfrak{P}\times_\G (I/I^2)^\vee]$$
on $\mathfrak{P}\times_\G W$ which is quasi-isomorphic to $\mathbb{R}T_\varrho$.
Denote by
$$F^\bullet=[F^0\lra F^1]$$
its pull-back to $\cC$ via $u$.

\begin{Lemma} Let $C$ be a geometric fiber of $\cC\stackrel\pi\lra \QmapW$. Then
$H^1(F^\bullet|_C)$ is a torsion sheaf on C.
\end{Lemma}
\begin{proof} Let $B\subset C$ be the set of base points of the stable quasimap. On $C\setminus B$ the complex $F^\bullet$ is quasi-isomorphic to
the pull-back of the tangent bundle of the (smooth) stable locus $W^s$ via $u|_{C\setminus B}$.
\end{proof}

\begin{Thm} Under the assumption that $W$ has at most lci singularities, the relative obstruction theory
$\left (R^\bullet\pi_*(u^*\mathbb{R}T_\varrho)\right )^\vee$ is perfect.
\end{Thm}

\begin{proof} We need to check that $R^\bullet\pi_*(F^\bullet)$ is perfect, of amplitude in $[0,1]$.

Using the Lemma, a spectral sequence computation shows that the second hypercohomology group
$$\mathbb{H}^2(F^\bullet |_C)$$ vanishes. The assertion now follows by standard arguments.
\end{proof}

\begin{Rmk}
Since $\BunG$ is a smooth Artin stack, it follows
that there is an absolute perfect obstruction theory on $\QmapW$,  
$$E^\bullet=Cone(f)[-1],$$ the shifted cone of the composite morphism
$$f:E_\mu^\bullet\lra\mathbb{L}_\mu\lra \mu^*\mathbb{L}_{\BunG}[1].$$

Furthermore, since $\phi:\BunG\lra\fMgk$ is a smooth map between smooth Artin stacks (see \S 2.1), we also have
an induced perfect obstruction theory $E_\nu^\bullet$ relative to $$\nu:\QmapW\lra\fMgk,$$ 
fitting in a distinguished triangle
$$ E_\nu^\bullet\lra E_\mu^\bullet\lra\mu^*\mathbb{L}_\phi[1].
$$

One sees immediately that the absolute
obstruction theories corresponding to $E_\mu^\bullet$ and $E_\nu^\bullet$ coincide.
It is a well-known
fact (see, e.g. \cite[Proposition 3]{KKP}) that in this case all three perfect obstruction theories determine
the same virtual fundamental class $[\QmapW]^{vir}$.
\end{Rmk}

\begin{Rmk}\label{stablemaps} The arguments in this section show that any open DM substack of the Artin stack of quasimaps
to $\WmodG$ obtained by imposing a stability condition has the canonical relative obstruction theory (\ref{relobstruction})
over $\BunG$. For example, if the stability condition imposed is that the image of $u$ is contained in the stable locus $W^s$
(i.e., there are {\it no} base points) and that the line bundle 
$$\omega _C(\sum_{i=1}^kp_i) \ot {\mathcal L}_\theta^{\ot 3}$$
is ample, then we obtain the Kontsevich moduli stack $\Mgk(\WmodG,\beta)$ of stable maps to $\WmodG$.
The proof of Proposition \ref{smooth} shows that the obstruction theory is perfect in this case,
irrespective of the nature of singularities of $W\setminus W^s$. The induced relative obstruction theory
over $\fMgk$ is the usual obstruction theory of stable map spaces, so the virtual classes are again the same.

In fact, as we discuss in \S 7 below there is a one-parameter family of stability conditions interpolating between stable
maps and stable quasimaps, giving DM moduli stacks with perfect obstruction theories relative to $\BunG$ given by
\eqref{relobstruction}.

In particular, all these Deligne-Mumford stacks, 
together with their canonical perfect obstruction theories,
will agree over the open substacks parametrizing honest stable maps with {\it semistable} underlying curve.

\end{Rmk}

\begin{EG}\label{ci} The typical context where the results of this subsection apply is a generalization of the
case (suggested in \cite{MOP}) of complete intersections in $\PP^n$.

Namely,
let $V$ be a vector space with $\G$-action linearized by a character $\theta$ satisfying the assumptions in $\S 2.2$.
Let $E$ be a linear $\G$-representation and consider the $\G$-equivariant vector bundle $V\times E$ on $V$.
It induces a vector bundle $$\overline{E}:=V^s\times_\G E$$ on $Y:=\VmodG$.
Let $t\in\Gamma(V,V\times E)^{\G}$ be a regular $\G$-invariant section and let $$W:=Z(t)\subset V$$ be its
zero locus, which we assume to be not entirely contained in the unstable locus in $V$.
The section $t$ descends to a section $\overline{t}$ of $\overline{E}$ on $Y=\VmodG$, whose zero locus
$$X:=Z(\overline{t})\subset Y$$ is identified with $\WmodG$. We assume that $X$ is nonsingular (or, equivalently, that $W^s=W\cap V^s$
is nonsingular).

Since (after
forgetting the $\G$-action)
$t$ is identified with a  $\rm{rank}(E)$-tuple of functions on $V$, we deduce that $W$ is an affine complete intersection
subscheme of $V$ and therefore we have the moduli stacks $\QmapW$ carrying canonical perfect obstruction theories.

In particular, if $Y=\PP^n=\CC^{n+1}/\!\!/\CC^*$,  the above construction gives precisely
the nonsingular complete intersections $X\subset \PP^n$, with $W$ the affine cone over $X$.

Note, however, that for an arbitrary nonsingular projective subvariety $X\subset\PP^n$ the affine cone $W$
may have a worse than lci singularity at the origin, and the obstruction theory of $\mathrm{Qmap}_{g,k}(W/\!\!/\CC^*,\beta)$ will {\it not} be
perfect in general.

Marian, Oprea, and Pandharipande introduced a ``moduli space of stable quotients on a projective subvariety $X\subset\PP^n$" as
the closed substack of $\mathrm{Qmap}_{g,k}(\CC^{n+1}/\!\!/\CC^*,d)$ given by the ideal sheaf induced by the equations of
$X$ in $\PP^n$ (see \S 10.1 of \cite{MOP}). They also suggested there that this moduli space carries a natural perfect obstruction
theory when $X$ is
a complete intersection and asked if this holds for arbitrary $X$. The above discussion answers their question.
Indeed, it is easy to see that their moduli space is the disjoint union of $\mathrm{Qmap}_{g,k}(W/\!\!/\CC^*,\beta)$ over
all curve classes $\beta$ which are sent to $d$ via the map
$$\Hom(\Pic^{\CC^*}(W),\ZZ)\lra\Hom(\Pic^{\CC^*}(V),\ZZ)\cong \ZZ.$$

\end{EG}

\subsection{Dependence on GIT presentation}

Given a reductive group $\G$, a quasiprojective scheme 
$W$ equipped with a $\G$-action, and a $\G$-linearized ample line bundle $L$, we can consider the pair of Artin stacks
$$\cA = [W/\G], \cA^{ss} = [W^{ss}/G],$$
where $[W/\G]$ denotes the stack quotient and $W^{ss}$ denotes the locus of semistable point in $W$ with respect to $L$.  
We say that the data $(\G, W, L)$ is a GIT presentation for the pair $(\cA, \cA^{ss})$.

One can ask to what extent the moduli stack of quasimaps depends on the choice of GIT presentation.  
It is easy to see that if one presentation satisfies the conditions stated at the beginning of this section, so does any other.  
Also, the space of curve classes is simply the dual of $\mathrm{Pic}(\cA)$ so is independent of presentation.  
The following proposition is then immediate from the results of this section.

\begin{Prop}  The moduli stack $\mathrm{Qmap}_{g,k}(\WmodG,\beta)$ 
and its perfect obstruction theory only depend on the underlying pair $(\cA, \cA^{ss})$.
\end{Prop}

Notice that this is weaker than claiming that the theory of quasimaps depends on either the GIT quotient or the stack quotient alone. 
 Given an Artin stack $\cA$ with reductive stabilizers, an approach to defining 
 Gromov-Witten invariants intrinsically to $\cA$ has been initiated in \cite{FTT}. 
  It would be interesting to compare their construction with ours.

\section{Quasimaps to quotients of vector spaces}\label{vspace}
In this section we discuss an alternative construction of
the stable quasimaps moduli in the important case when $W=V$ is a vector space.
This is a situation general enough to include many interesting instances of GIT quotient targets, such as toric varieties
and flag manifolds of type $A$. One reason this approach could be potentially useful is that it 
comes automatically with an easy and concrete
description of the virtual fundamental class: the moduli space is realized as the zero locus of a section
of a vector bundle on an appropriate {\it nonsingular} Deligne-Mumford stack and the
virtual class is simply given by the Fulton-MacPherson refined top Chern class of this vector bundle.

The construction presented here is a straightforward extension of
the one used in \cite{CK} for the moduli of stable toric quasimaps. Note also
that it recovers the moduli of stable quotients with its virtual class
by an approach different than that of \cite{MOP}.

\subsection{A generalized Euler sequence}
Recall our setup.
Let $V$ be a finite dimensional $\CC$-vector space equipped with a $\G$-action via
a representation $\G\lra \mathrm{GL}(V)$. We fix a linearization
$\theta\in\chi(\G)=\Pic^\G(V)$
satisfying
%\begin{enumerate}

\phantom{x}
$(i)$
$\emptyset\neq V^s=V^{ss}$.

$(ii)$ $\G$ acts freely on $V^s$. 

\phantom{x}

%\end{enumerate}
Let $\rho:V^s\lra\VmodG=V^s/\G$ be the projection.

It is useful to note the following presentation of the tangent bundle of $\VmodG$.
Namely, the infinitesimal action of $\G$ on $V^s$ induces an exact sequence of vector bundles
$$0 \lra V^s \times \mathfrak{g} \lra TV^s = V^s \times V \lra \rho^*T(\VmodG) \lra 0$$
where $\mathfrak{g}$ is the Lie algebra of $\G$.  This sequence is
$\G$-equivariant with respect to the adjoint action on $\mathfrak{g}$, so descends to an exact sequence
\begin{equation}\label{Eulerseq}
0\lra V^s\times_\G\mathfrak{g}\lra V^s\times_\G V\lra T_{\VmodG}\lra 0.
\end{equation}
This construction generalizes the Euler sequence presentation of the tangent bundle of projective space.

\subsection{Quasimaps from curves to $\VmodG$}

To give a quasimap of class $\beta$ to $\VmodG$ on a curve $C$ is the same as
giving a principal $\G$-bundle $P$ on $C$ of degree $\beta_P=\beta$ and a global section
$$u \in \Gamma (C, \cV_P )$$ of the induced vector bundle $\cV_P:=P\times_\G V$,
satisfying the generic nondegeneracy condition
\eqref{generic-ndeg}.

Let $k, g$ be fixed with $2g-2+k\geq 0$, and let $\beta$ be a $L_\theta$-effective class.
We have the moduli stack $$\mathrm{Qmap}_{g,k}(\VmodG,\beta)$$ of $k$-pointed, genus $g$ stable quasimaps of class
$\beta$ to $\VmodG$. 

\begin{Thm}\label{Thm2} 
$\mathrm{Qmap}_{g,k}(\VmodG,\beta)$ is a finite type Deligne-Mumford stack,
proper over the affine quotient $V/_{\mathrm{aff}}\G=\Spec (\mathrm{Sym}(V^\vee)^\G)$, and it
carries a perfect obstruction theory. 
\end{Thm}

\begin{proof} Given the boundedness results we established in $\S \ref{bdd}$, the argument is almost identical
to the one given in \cite{CK} for the special case of toric varieties.

There are forgetful morphisms of stacks making a commutative diagram
$$\begin{array}{ccc}
\QmapV& =&\QmapV\\ \\
\mu\downarrow & &\downarrow\nu \\ \\
\BunG &\lra &\fMgk
\end{array}$$
By Corollary \ref{boundedness-curve}, the image of $\nu$ is contained in an open and closed substack
of finite type $\mathfrak{S}\subset\fMgk$.
By Corollary \ref{boundedness-curve} and Theorem \ref{boundedness-bundle}, the image of $\mu$ is contained
in an open and closed substack  of finite type $\mathfrak{B}un_{\G,\beta}\subset\BunG$, lying over $\fS$.

On the universal curve $\mathfrak{C}\stackrel\pi\lra\BunGbeta$ we have the universal principal bundle $\mathfrak{P}$
and the induced universal vector bundle $\mathcal{V}_\mathfrak{P}=\mathfrak{P}\times_\G V$. 
As in \S \ref{construction}, after replacing $\BunGbeta$ by an open substack, we may assume that $\pi$ is a projective morphism.
Let $\cO_\mathfrak{C}(1)$
be a $\pi$-relatively ample line bundle and fix a global section $0\lra\cO_\mathfrak{C}\lra\cO_\mathfrak{C}(1)$. 
Define a sheaf $\mathcal{G}_m$ on $\mathfrak{C}$ by the exact sequence
$$0\lra\cV_\mathfrak{P}\lra\cV_\mathfrak{P}(m)\lra\mathcal{G}_m\lra 0.$$
By Corollary \ref{regularity} there is an $m>\!\!>0$, which we will fix from now on, such that
$$\pi_*(\cV_\mathfrak{P}(m))\;\;\;\mathrm{and}\;\;\ \pi_*(\mathcal{G}_m)$$
are vector bundles on $\BunGbeta$.

Let $\mathfrak{X}$ be the total space of the vector bundle $\pi_*(\cV_\mathfrak{P}(m))$. It is a smooth
Artin stack of finite type, whose $\CC$-points
parametrize tuples
\begin{equation}\label{ambient}
(C,p_1,\dots, p_k, P, u\in\Gamma(C,\cV_P(m))).
\end{equation}
Let $\cX^\circ$ be the open substack in $\mathfrak{X}$ determined by imposing the requirements
\begin{itemize}
\item $u$ sends all but possibly finitely many points of $C$ to $V^s$; here we identify the fibers
of $\cV_P(m)$ with $V$.\footnote{Apriori, this identification only makes sense locally, in a trivialization $\{U_i, \varphi_i\}$ of $\cV_P(m)$ on $C$.
However, one checks that the local $V^s\times U_i$ glue together into a global open subscheme $\cV_P(m)^s\subset\cV_P(m)$ under the transition functions
of $\cV_P(m)$. This is so since the $\G$-invariant locus $V^{ss}=V^s$ is also invariant under the scaling action by the homotheties of $V$. In turn,
this invariance follows easily from the definition of $\theta$-semistable points and the fact that the group of homotheties is identified with the center of $GL(V)$, and therefore
its action on $V$ commutes with the action of $\G$. }
\item the line bundle $\omega _C(\sum_{i=1}^kp_i) \ot {\mathcal L}^\epsilon $ is ample for every rational $\epsilon > 0$, where
${\mathcal L}=P\times_\G L_\theta$.
\end{itemize}
By stability, $\cX^\circ$ is a smooth Deligne-Mumford stack. It comes with a projection $p:\cX^\circ\lra\BunGbeta$. The vector
bundle
\begin{equation}\cF:=p^*(\pi_*(\mathcal{G}_m))\label{bundle}\end{equation}
on $\cX^\circ$ has a tautological section $s$ induced by the maps
$$\Gamma(\mathfrak{C},\cV_\mathfrak{P}(m))\lra\Gamma(\mathfrak{C},\mathcal{G}_m).$$
Its zero locus $Z(s)$
is the closed substack parametrizing tuples as in \eqref{ambient}, but with $u\in\Gamma(C,\cV_P)\subset\Gamma(C,\cV_P(m))$.

Finally, $\QmapV$ is identified with the open substack of $Z(s)$ obtained by
imposing the strong nondegeneracy condition that the base-points are away from nodes
and markings, and is therefore a Deligne-Mumford stack of finite type.

Properness has been proved more generally in the previous section.

Finally, it is easy now to describe the relative and absolute obstruction theories.
For this, we note first that the order in which the last two conditions were imposed in the construction
of $\QmapV$ may be in principle reversed. Namely, there
is an open substack $\cE\subset \cX^\circ$ (which is necessarily DM and smooth) such that $\QmapV$ is cut out in $\cE$
as the zero locus $Z(s)$ of the tautological section of the bundle (\ref{bundle}) {\it restricted} to $\cE$.

Let
$$\mathfrak{C}\stackrel\pi\lra\QmapV$$
be the universal curve. The above description gives an absolute perfect obstruction theory 
$$[\cF^\vee|_\mathrm{Qmap}\stackrel{ds^\vee}\lra \Omega^1_{\cE}|_\mathrm{Qmap}]$$
on $\QmapV$,
whose virtual class is the refined top Chern class of $\cF^\vee|_\mathrm{Qmap}$.

As explained in  \cite[\S 5]{CK}, the induced
relative obstruction theory over $\BunG$ is
$$[\cF^\vee|_\mathrm{Qmap}\stackrel{ds^\vee}\lra \Omega^1_{\cE/\BunGbeta}|_\mathrm{Qmap}]\;\;
\stackrel{\mathrm{qis}}\sim \;\;(R^\bullet\pi_*(\cV_\mathfrak{P}))^\vee$$
and the relative obstruction theory over $\fMgk$ is
$$(R^\bullet\pi_*(\cQ))^\vee,$$
where $\cQ$ is defined by the ``Euler sequence" on $\mathfrak{C}$
$$0\lra\mathfrak{P}\times_G\mathfrak{g}\lra\cV_\mathfrak{P}\lra\cQ\lra 0$$
induced by \eqref{Eulerseq}.
\end{proof}

\begin{Rmk} If $W$ is a $\G$-invariant affine subvariety of $V$, then 
we have the induced GIT quotient $\WmodG$ (using the same linearization $\theta$).
By restricting only to quasimaps
$((C,p_i),P,u)$ for which
$$u(C)\subset P\times_\G W\subset P\times_\G V$$
we obtain a closed substack  of $\QmapV$. It is obvious that this substack is identified with
$$\coprod_{\widetilde{\beta}\mapsto\beta}\mathrm{Qmap}_{g,k}(W/\!\!/\G,\widetilde{\beta}),$$ with 
the disjoint union over
all curve classes $\widetilde{\beta}$ which are mapped to $\beta$ via the morphism
$$\Hom(\Pic^{\G}(W),\ZZ)\lra\Hom(\Pic^{\G}(V),\ZZ)$$
induced by the pull-back $\Pic^\G(V)\lra\Pic^\G(W)$. 
Hence Theorem \ref{Thm2} provides a different proof of the fact that $\mathrm{Qmap}_{g,k}(W/\!\!/\G,\widetilde{\beta})$
is a Deligne-Mumford stack of finite type for these targets $\WmodG$.
\end{Rmk}

%%%%%%%%%%%%%%%%%%%%%
\section{Quasimap invariants}
\subsection{Descendant invariants} Throughout this subsection and the next we assume that $\WmodG$ is projective.
Since the moduli spaces $\QmapW$ are proper Deligne-Mumford stacks with perfect obstruction theories, a system of invariants is obtained
by integrating natural cohomology classes against the virtual class. Precisely, we have the following canonical structures:

\begin{itemize}
\item Evaluation maps
$${ev}_i :\QmapW\lra \WmodG, \;\;\ i=1,\dots, k$$
at the marked points. These are well-defined since base-points cannot occur at markings.

\item Cotangent line bundles
$$M_i:= s_i^*(\omega_{\cC/\QmapW}),\;\; i=1,\dots, k$$
where $\cC$ is the universal curve, $\omega_{\cC/\QmapW}$ is the relative dualizing sheaf (which is a line bundle), and
$$s_i:\QmapW\lra\cC$$
are the universal sections. We denote
$${\psi}_i:=c_1(M_i).$$
\end{itemize}
\begin{Def} The descendant quasimap invariants are
$$\langle \tau_{n_1}(\gamma_1),\dots,\tau_{n_k}(\gamma_k)\rangle_{g,k,\beta}^{quasi}:=
\int_{[\QmapW]^{\mathrm{vir}}}\prod_{i=1}^k{\psi}_i^{n_i}{ev}_i^*(\gamma_i),$$
with $\gamma_1,\dots ,\gamma_k\in H^*(\WmodG,\QQ)$ and
$n_1,\dots,n_k$  nonnegative integers.

\end{Def}

As explained in \cite{CK}, the quasimap invariants satisfy the analogue of the Splitting Axiom in Gromov-Witten theory.
Furthermore, there is a natural map
$$f:\QmapW\lra\Mgk$$
which forgets the principal $\G$-bundle and the section, and contracts the unstable components of the underlying pointed curve.
We may define {\it quasimap classes} in $H^*(\Mgk,\QQ)$ by
$$f_*(\prod_{i=1}^k{\psi}_i^{n_i}{ev}_i^*(\gamma_i)),$$
which by the splitting property give rise to a {\it Cohomological Field Theory} on $H^*(\WmodG,\QQ)$.

On the other hand, the universal curve over $\QmapW$
is {\it not} isomorphic to $\mathrm{Qmap}_{g,k}(\WmodG,\beta)$, and in general there are no maps
$$\mathrm{Qmap}_{g,k+1}(\WmodG,\beta)\lra\QmapW$$
which forget one marking. Note that the corresponding maps for the Kontsevich moduli spaces of stable
maps are crucially used in the proofs of the string and divisor equations for descendant Gromov-Witten invariants.
Some examples showing that the exact analogue of the
string and divisor equations fail for quasimap invariants of certain
non-Fano toric varieties are given in
\cite{CK}. However, the quasimap and Gromov-Witten theories of Grassmannians are shown to coincide in
\cite{MOP}; the same was conjectured to hold for Fano toric varieties, see \cite{CK}, and is proved in \cite{CK2}.
Furthermore, for general targets
we expect that the two theories  are related by ``wall-crossing" formulas involving 
the $\epsilon$-stable quasimaps described in \S\ref{e-stable} below (this will be addressed elsewhere).
Consequently, the quasimap invariants should still satisfy some modified versions of the equations. At the moment
it is not yet clear to us if reasonable general formulae can be written down.

\subsection{Twisted invariants}\label{twisted} In Gromov-Witten theory, twisted invariants have been introduced and
studied by Coates and Givental, \cite{CG}. Their counterparts in quasimap theory are easily obtained.

Let
$$\cC \stackrel\pi\lra\QmapW$$ be the universal curve, with universal principal $\G$-bundle $\mathfrak{P}$ on it.
Let
$$\tilde{u}:\fP\lra W$$ be the universal $\G$-equivariant map and let
$$u:\cC\lra\fP\times_\G W$$
be the induced universal section.
Let $E$ be any $\G$-equivariant vector bundle on $W$ (for example, we could take a linear
$\G$-representation $E$ and, by slight abuse of notation, denote also by $E$ the associated
$\G$-equivariant vector bundle $W\times E$) and let $\overline{E}=[E|_{W^s}/G]$ be the induced vector bundle on $\WmodG$. Then
$\fP\times_\G E$ is a vector bundle on $\fP\times_\G W$ and we consider the vector bundle
$$E_{g,k,\beta}:=u^*(\fP\times_\G E)$$
on the universal curve.

\vskip .2in

\noindent{\it Claim}: $R^\bullet\pi_*E_{g,k,\beta}:=[R^0\pi_*E_{g,k,\beta}]-[R^1\pi_*E_{g,k,\beta}]$ 
is an element in $$K^\circ(\QmapW),$$ the $K$ group of vector bundles on
$\QmapW$.
\begin{proof} Let $\cO(1)$ be a $\pi$-relatively ample line bundle on $\cC$. For $m>\!\!>0$, we have a surjection
$$B\lra E_{g,k,\beta}(m)\lra 0,$$
with $B$ a trivial vector bundle. The kernel, call it $A$, is also a vector bundle on $\cC$, and there is an exact sequence
$$0\lra A(-m)\lra B(-m)\lra E_{g,k,\beta}\lra 0.
$$
Since
$$R^0\pi_*(A(-m))=R^0\pi_*(B(-m))=0,$$
we have a complex of vector bundles
$$R^1\pi_*(A(-m))\lra R^1\pi_*(B(-m))$$
whose cohomology is precisely $R^\bullet\pi_*E_{g,k,\beta}$.
\end{proof}

Let $\CC^*$ act trivially on $\mathrm{Qmap}=\QmapW$, with equivariant parameter $\lambda$, so that
$$H^*_{\CC^*}(\mathrm{Qmap},\QQ)\cong H^*(\mathrm{Qmap},\QQ)\otimes_\QQ \QQ[\lambda].$$
Similarly, the rational Grothendieck group of $\CC^*$-equivariant vector bundles is
$$K^\circ_{\CC^*}(\mathrm{Qmap})\cong K^\circ(\mathrm{Qmap})\otimes \QQ[\lambda,\lambda^{-1}].$$

Now fix an invertible multiplicative class $c$, that is, a homomorphism
$$c:K^\circ_{\CC^*}(\mathrm{Qmap})\lra U(H^*(\mathrm{Qmap},\QQ)\otimes\QQ[\lambda,\lambda^{-1}])$$
to the group of units in the localized $\CC^*$-equivariant cohomology ring. In addition, fix a $\G$-equivariant
bundle $E$ on $W$. Let $\CC_\lambda$ denote the 1-dimensional representation of $\CC^*$ with weight $\lambda$.
Applying the construction described earlier in this subsection to the $\G\times \CC^*$-equivariant
bundle $E\otimes\CC_\lambda$ (for the trivial action of $\CC^*$ on $W$) yields an element
$$R^\bullet\pi_*E_{g,k,\beta}(\lambda)\in K^\circ(\mathrm{Qmap})\otimes \QQ[\lambda,\lambda^{-1}].$$

\begin{Def} The $(E, c)$-twisted quasimap invariants are
$$\langle \tau_{n_1}(\gamma_1),\dots,\tau_{n_k}(\gamma_k)\rangle_{g,k,\beta}^{quasi, (c,E)}:=
\int_{[\mathrm{Qmap}]^{\mathrm{vir}}}c(R^\bullet\pi_*E_{g,k,\beta}(\lambda))\prod_{i=1}^k{\psi}_i^{n_i}{ev}_i^*(\gamma_i).$$
\end{Def}

By definition, the twisted invariants lie in $\QQ[\lambda,\lambda^{-1}]$.

We discuss next an important example of this construction.
A typical choice of multiplicative class for twisting is the equivariant
Euler class. For an ordinary bundle
$F$ of rank $r$ on a space $Y$, with Chern roots $f_1,\dots,f_r$, the Euler class of $F\otimes\CC_\lambda$ is
$$e(F\otimes\CC_\lambda)=\prod_{i=1}^r(f_i+\lambda)=\lambda^r+\lambda^{r-1}c_1(F)+\dots+c_r(F),$$
i.e., a version of the Chern polynomial of $F$.

Suppose now that we are in the situation described in Example \ref{ci}: $V$ is a vector space, $E$ is the bundle
$V\times E$ coming from a linear $\G$-representation, $W=Z(t)$ is the zero locus of a regular section
$t\in\Gamma(V,E)^\G$, and $W\cap V^s$ is nonsingular. For equivariant curve classes $\beta\in \Hom(\Pic^\G(V),\ZZ)$
and $\tilde{\beta}\in\Hom(\Pic^\G(W),\ZZ)$ we write $\tilde{\beta}\mapsto\beta$ if $\tilde{\beta}$ is sent to $\beta$ under the
natural map between the duals of equivariant Picard groups. Put
$$\QmapW:=\coprod_{\tilde{\beta}\mapsto\beta}\mathrm{Qmap}_{g,k}(\WmodG,\tilde{\beta}).$$
There is an induced closed embedding of stacks
$$i:\QmapW\lra\QmapV.$$

\begin{Prop}\label{pushforward} Assume that $R^1\pi_*E_{g,k,\beta}=0$. Then
\begin{equation}\nonumber\begin{split}
i_*[&\QmapW]^{\mathrm{vir}} =e(R^\bullet\pi_*E_{g,k,\beta})\cap[\QmapV]^{\mathrm{vir}}\\&=
(e(R^\bullet\pi_*E_{g,k,\beta}(\lambda))\cap[\QmapV]^{\mathrm{vir}})|_{\lambda=0}.\end{split}\end{equation}
\end{Prop}
\begin{proof} The assumed vanishing implies that $R^0\pi_*E_{g,k,\beta}$ is a vector bundle. The second equality is now
immediate from the definitions, while the first follows from \cite{KKP} by the same argument as the one given there for moduli
of stable maps.
\end{proof}

As we already mentioned in the proof, Proposition \ref{pushforward} holds for stable maps, however, the required vanishing
is true only for $g=0$. It is a very important unsolved problem in Gromov-Witten theory to find a useful expression for the push-forward of
the virtual class in the proposition when $g\geq 1$. As observed in \cite{MOP} (for the special case of complete intersections in $\PP^n$), one
can do slightly better in quasimap theory. We will say that a character $\eta\in\chi(\G)$ is {\it positive} with respect to our given linearization $\theta$
if $V^s(\theta)=V^s(\eta)=V^{ss}(\eta)$ (in particular, the line bundle induced by $\eta$ on $\VmodG$ is ample), and we will say it is {\it semi-positive} if 
$V^s(\theta)\subset V^{ss}(\eta)$.

\begin{Prop}\label{vanishing} The vanishing $R^1\pi_*E_{g,k,\beta}=0$, hence the conclusion of
Proposition \ref{pushforward}, holds in the following cases:

$(i)$ $g=0$, $k$ arbitrary, and $E=\oplus_{i=1}^r\CC_{\eta_i}$, with each $\eta_i$ semi-positive.

$(ii)$ $g=1$, $k=0$, and $E=\oplus_{i=1}^r\CC_{\eta_i}$, with each $\eta_i$ positive.

\end{Prop}
\begin{proof} Let $C$ be a geometric fiber of the universal curve, let $P$ be the corresponding
principal bundle, and let $u:C\lra P\times_\G V$ be the corresponding section. The restriction to any irreducible
component $C'$ of $C$ of the vector bundle
$$u^*(P\times_\G(V\times E))$$
decomposes as the direct sum of line bundles
$$\oplus_iu^*(P|_{C'}\times_\G(V\times \CC_{\eta_i}))\cong\oplus_i(P|_{C'}\times_\G\CC_{\eta_i}).$$

In case $(i)$ each irreducible
component $C'$ of $C$ is a rational curve. The semi-positivity assumption implies that each of the line bundles
has nonnegative degree on $C'$ by the argument of Lemma \ref{effective}. We deduce that
$$H^1(C,u^*(P\times_\G(V\times E)))=0,$$
from which the vanishing follows.

In case $(ii)$ the underlying curve of a stable quasimap is either an irreducible elliptic curve,
or a cycle of rational curves (this last case includes the cycle of length one, i.e., an irreducible nodal curve
of arithmetic genus one) . Since there are no special line bundles of positive degree on such curves,
the vanishing follows from the positivity of the $\eta_i$'s.
\end{proof}

\begin{Rmk} It is clear that
in $(i)$ it suffices to assume that $E$ is a representation such that
the bundle $V\times E$ is generated by $\G$-equivariant global sections.
\end{Rmk}

%%%%%%%%%%%%%%%%%%%%%%%%%%%%%%%%%%%%%%%%%%%%%%%%%%%%%%%%%%%%%%%%%%%%%
%%%%%%%%%%%%%%%%%%%%%%%%%%%%%%%%%%%%%%%%%%%%%%%%%%%%%%%%%%%%%%%%%%%%%

\subsection{Invariants for noncompact $W/\!\!/\G$}\label{noncompact}

Recall that when $\WmodG$ is not projective we have only a proper morphism to the affine quotient
$$\QmapW\lra \mathrm{Spec}(A(W)^\G).$$

As is the case in Gromov-Witten theory, even though the integrals against the virtual class are not well-defined
due to the lack of properness of $\QmapW$, there are many situations for which one can get an interesting theory of
{\it equivariant} invariants via the virtual localization formula of \cite{GP}. This happens when there is a
torus $\bS\cong (\CC^*)^m$ with an action on $W$, commuting with the given $\G$-action, and such that the
fixed locus of the induced
$\bS$-action on $\QmapW$ is proper.

The $\bS$-action on $\QmapW$, is given on $\CC$-points by
$$s\cdot ((C,p_1,\dots,p_k),P,u)=((C,p_1,\dots,p_k),P,s\circ u).
$$
Here we view $s\in\bS$ as an automorphism of $W$, giving rise to a
$C$-automorphism of the fibration $P\times_\G W\lra C$, still denoted by $s$.

Note that the map
$$\mu:\QmapW\lra\BunGbeta$$
is $\bS$-equivariant, where $\bS$ acts trivially on $\BunGbeta$.

From its description in (\ref{relobstruction}), it is clear that the
$\mu$-relative obstruction theory is $\bS$-equivariant, hence the same holds for the absolute
obstruction theory as well.

Furthermore, the stack $\QmapW$ admits an $\bS$-equivariant closed embedding into a smooth Deligne-Mumford
stack. When $W=V$ is a vector space, this is shown already by the construction in \S \ref{vspace}; in general, one can 
essentially deduce it from this case using Proposition \ref{embedding} 
(some care is need to handle the fact that $V$ may have strictly semistable points and stable points with nontrivial stabilizer).

It follows from \cite{GP} that each component $F$ of the fixed point locus
$\QmapW^{\bS}$ has a virtual fundamental class and a virtual normal bundle, whose Euler class
is invertible in $H^*_\bS(F,\QQ)\otimes\QQ(\lambda_1,\dots,\lambda_m)$.
Here $\lambda_1,\dots, \lambda_m$ denote the equivariant parameters, so that
$$H^*_{\bS}(\mathrm{pt},\QQ)=H^*(B\bS)\cong\QQ[\lambda_1,\dots, \lambda_m].$$

Finally, as mentioned above, we will make the following

\begin{Ass} The $\bS$-fixed closed substack $\QmapW^{\bS}$ is proper.
\end{Ass}
For example, this will always hold when the fixed point locus for the induced $\bS$-action
on the affine quotient $\mathrm{Spec}(A(W)^\G)$ is proper -- and therefore a finite set of points.
This follows immediately from the fact that the map $\eta:\QmapW\lra \mathrm{Spec}(A(W)^\G)$ is
$\bS$-equivariant and proper.

We can then define invariants as sums of equivariant residues via the virtual localization formula:
for equivariant cohomology classes $\gamma_1,\dots ,\gamma_k\in H^*_{\bS}(\WmodG,\QQ)$ and
nonnegative integers $n_1,\dots,n_k$, the equivariant descendant quasimap invariant is
$$\langle \tau_{n_1}(\gamma_1),\dots,\tau_{n_k}(\gamma_k)\rangle_{g,k,\beta}^{quasi}:=
\sum_F\int_{[F]^{\mathrm{vir}}}\frac{i_F^*(\prod_{i=1}^k{\psi}_i^{n_i}{ev}_i^*(\gamma_i))}{e(N_F^{\mathrm{vir}})},$$
the sum over the connected components $F$ of the fixed point locus
$\QmapW^{\bS}$, with equivariant embeddings
$$i_F:F\lra\QmapW.$$
By the above definition, the quasimap invariants lie in $\QQ(\lambda_1,\dots,\lambda_m)$.

Similarly, we may also define twisted invariants as in \S\ref{twisted}, using
invertible multiplicative classes of $\bS$-equivariant
vector bundles.

The main examples we have in mind of noncompact targets with a well-defined theory of
quasimap invariants are found among quiver varieties.

\begin{EG} {\bf Nakajima quiver varieties}.
Let $\Gamma$ be an oriented graph on a finite set of vertices $S$ with
edge set $E$, equipped with
source and target maps
$$i, o: E \rightarrow S.$$  Suppose we are given two dimension vectors
$$\overrightarrow{v}, \overrightarrow{u} \in (\mathbb{Z}_{\geq 0})^{S}.$$
The Nakajima quiver variety associated to this data is defined as follows.
For each vertex $s \in S$, we fix vector spaces $V_s$, $U_s$ of dimension
$v_s$ and $u_s$ respectively.  Consider the associated affine space
$$\mathbb{H} = \bigoplus_{e \in E} \left( \Hom(V_{i(e)}, V_{o(e)}) \oplus
\Hom(V_{o(e)}, V_{i(e)}) \right)\bigoplus_{s\in S} \left(\Hom(U_s, V_s)
\oplus \Hom(V_s, U_s)\right).$$

Given an element $(A_e, B_e, i_s, j_s)_{e\in E, s\in S} \in \mathbb{H}$,
and a vertex $s \in S$, we can associate the following endomorphism
$$\phi_s = \sum_{e, i(e) = s} [A_e,B_e] - \sum_{e, o(e) = s}[A_e,B_e] +
i_s\circ j_s \in \mathrm{End}(V_s).$$
Let $\G = \prod_{s\in S} GL(V_s)$ act on $\mathbb{H}$ with polarization
given by the determinant character.  We are interested in the GIT quotient
of the $\G$-invariant affine variety $W$ defined by the equations $\phi_s = 0$:
$$\mathcal{M}(\overrightarrow{v}, \overrightarrow{u}) = \{h \in \mathbb{H}
| \phi_s = 0 \textrm{ for all } s\} /\!\!/\G.$$

These quotients are typically noncompact.  However, if one considers the
$\bS=\mathbb{C}^*$-action defined by scaling $A_e$ and $i_s$, it is easy to
see that it preserves $W$ and $(\Spec \mathbb{C}[W]^\G)^{\mathbb{C}^*}$ is
compact (see, for instance \cite{nakajima}).

In particular, this gives us a rich source of examples where the
noncompact invariants make sense.  
\end{EG}

For other applications of quasimap theory with noncompact targets the reader is referred to the papers
\cite{Kim3}, \cite{KL}.

%%%%%%%%%%%%%%%%%%%%%%%%%%%%%%%%%%%%%%%%%%%%%%%%%%%%%%%%%%%%%%%%%%%%%%%%%%%%%%%
\section{Variants and applications}

\subsection{$\epsilon$-stable quasimaps}\label{e-stable}

As we have observed in \S 4 (see in particular Remarks \ref{general} and \ref{stablemaps}) one may view the moduli spaces
of stable quasimaps and stable maps with target $\WmodG$ as two instances of the same construction: they are 
(relatively) proper, finite type, Deligne-Mumford substacks of the Artin stack, locally of finite type parametrizing
prestable quasimaps to $\WmodG$, obtained by imposing a stability condition. Furthermore, one 
may think loosely of passing from stable maps to stable quasimaps as a process in which, by changing the stability condition,
rational tails are replaced with base points that keep track of their degrees (cf.
the proof of Proposition \ref{propernessProp}). It is natural to try to do this sequentially, removing first rational
tails of lowest degree, then those of next lowest degree and so on.

In the case of target $\PP^n$, this loose interpretation is literally true, and the result of
the procedure is the factorization of a natural morphism from the moduli of stable maps to the
moduli of stable quasimaps into a sequence of blow-downs to intermediate moduli spaces.
These intermediate spaces correspond to stability conditions depending on a rational parameter $\epsilon$.
This was first noticed and proved some time ago 
by Musta\c t\u a and Musta\c t\u a in \cite{MM1, MM2}. 

More recently, using moduli of stable quotients,
Toda (\cite{Toda}) has extended the story to Grassmannian targets $\mathrm{G} (r, n)$.
In this case, there are typically no morphisms between the various 
$\epsilon$-stable moduli spaces, so one obtains an instance of the general wall-crossing
phenomenon when varying the stability condition.

In this subsection we treat the variation of stability for general targets $\WmodG$.
Recall that a quasimap 
is said to be prestable if the base points are away from nodes and markings.

\begin{Def}\label{length}
The {\em length} $\ell(x)$ at a point $x\in C$ of a prestable quasimap $((C,p_i),P,u)$ to $\WmodG$ is defined by
$$\ell(x):= \min \left\{ \frac{(u^*s)_x}{m} \ | \;\; \ s\in H^0(W, L_{m\theta})^\G,\; u^*s\not\equiv 0, \;\; m>0\right\},$$
where $(u^*s)_x$ is the coefficient of the divisor $(u^*s)$ at $x$.
 \end{Def}
 
 The following properties follow easily:
 \begin{itemize}
\item For every $x\in C$ we have
 $$\beta(L_\theta)\geq\ell(x)\geq 0$$ 
 and $\ell(x)>0$ if and only if $x$ is a base point of the quasimap.
This is because the scheme-theoretic unstable locus $W^{us}$ is the subscheme defined by the
ideal $J_{W^{us}}\subset A(W)$ generated by $\{s \; |\; s\in H^0(W, L_{m\theta})^\G,\; m>0\}$.
 
 \item If $H^0(W,L_\theta)^\G$ generates $\oplus_{m\geq 0}H^0(W,L_{m\theta})^\G$ as an algebra over $A(W)^\G$
(for example, if the relatively ample line bundle $\cO(\theta)$ induced by $L_\theta$ on $\WmodG$ is relatively very ample over the affine quotient) then we can also write
 \begin{equation}\label{rescaled}\ell(x)= \min \left\{ (u^*s)_x \ | \;\; \ s\in H^0(W, L_{m\theta})^\G,\; u^*s\not\equiv 0,\;\; m>0\right\}.\end{equation}
 Alternatively, consider the ideal sheaf $\mathcal{J}$ of the closed subscheme $P\times_\G W^{us}$ of $P\times_\G W$.
 Then it is clear that $\ell(x)$ from \eqref{rescaled} satisfies
 \begin{equation}\label{multiplicity}\ell(x)=\mathrm{length}_x(\mathrm{coker}(u^*\mathcal{J}\lra \cO_C)),\end{equation}
 which may be viewed as the order of contact of $u(C)$ with the unstable subscheme $P\times_\G W^{us}$ at $u(x)$.
 \end{itemize}
 
 Let $((C,p_i),P,u)$ be a prestable quasimap of class $\beta$. Let $B\subset C$ be the base locus. Let 
 $$[u]:C\setminus B \ra \WmodG$$ be the induced map. By the prestable condition and the projectivity
 of $\WmodG\lra\Spec(A(W)^\G)$, $[u]$ extends to a regular map
 $$[u_{reg}]:C\lra \WmodG.$$
 Let $P_{reg}$ be the principal $\G$-bundle on $C$ which is obtained as the pull-back of $W^s\lra\WmodG$ via $[u_{reg}]$.
 Let $u_{reg} : C\lra P_{reg}\times_G W$ be the induced section. The data 
 $$((C,p_i),P_{reg}, u_{reg})$$
 is also a quasimap to $\WmodG$, so it has a class $\beta_{reg}\in \Hom(\Pic^\G(W),\ZZ)$.  
 
 \begin{Lemma}\label{degree-splitting} We have 
\begin{equation}\label{split}(\beta-\beta_{reg})(L_\theta)=\sum_{x\in B}\ell(x).\end{equation}
 In particular, $\beta-\beta_{reg}$ is $L_\theta$-effective.
 \end{Lemma}
 \begin{proof} It suffices to assume that $L_\theta$ descends to a relatively very ample line bundle on $\WmodG$, since
 the equality \eqref{split} is invariant under replacing $\theta$ by a positive multiple. But in this case the equality is immediate.
 \end{proof}

Fix a positive rational number $\epsilon$.

\begin{Def}\label{eDef} A prestable quasimap $((C,p_i),P ,u)$ is called {\em $\epsilon$-stable} if 
\begin{enumerate}
\item $\omega _C (\sum p_i ) \otimes \mathcal{L}_{\theta} ^\epsilon$ is ample.
\item  $\epsilon  \ell(x) \le 1$ for every point $x$ in $C$.
\end{enumerate}
\end{Def}

\begin{Rmk} The notion of $\epsilon$-stability depends obviously on $\theta$. However, this dependence 
is very simple under rescaling: 
for a positive integer $m$, the quasimap $((C,p_i),P ,u)$
is $\epsilon$-stable with respect to $\theta$ if and only if it is $(\frac{\epsilon}{ m})$-stable with respect to $m\theta$.

As a consequence, we may assume (and will assume from now on) that the polarization $\cO(\theta)$ on $\WmodG$ is 
relatively very ample over the affine quotient. Hence we can use \eqref{rescaled} and \eqref{multiplicity} as (equivalent) definitions of length.
In particular, $\ell(x)\in\ZZ_+$ for every point $x$ of $C$.
\end{Rmk}

Again, we list some immediate consequences of this definition.
\begin{itemize} 
\item The underlying curve of an $\epsilon$-stable quasimap 
may have rational components containing
only one special point (i.e., rational tails). However, any such component $C'$ must satisfy
\begin{equation}\label{tail}
\epsilon\deg(\cL_\theta|_{C'})>1.
\end{equation}
In addition, the degree of $\cL_\theta$ must be positive on rational components containing exactly
two special points.

\item A prestable quasimap $((C,p_i),P ,u)$ of some class $\beta$ is a stable quasimap (as in Definition \ref{stability})
if and only it is an $\epsilon$-stable quasimap for some $\epsilon \leq 1/\beta (L_\theta )$.

\item Assume that $(g,k)\neq(0,0)$. Then a prestable quasimap is a stable map to $\WmodG$ if and only 
if it is an $\epsilon$-stable quasimap for some $\epsilon >1 $. In the case $(g,k)=(0,0)$ the same is true
but with $\epsilon >2$.
\end{itemize}
Therefore, for the extremal values of $\epsilon$ we recover the notions of stable quasimaps and stable maps, respectively.

\begin{Prop}\label{finite autom} Let $\epsilon >0$ be fixed.
The automorphism group of an $\epsilon$-stable quasimap is finite and reduced.
\end{Prop}
\begin{proof} Assume first $(g,k)\neq (0,0)$. It suffices to check what happens on rational tails of the underlying curve $C$. Let $C'$ be a rational tail
and let $\beta_{C'}$ be the class of the induced quasimap on $C'$. Any automorphism of the quasimap must preserve
the map $[u_{reg}]$, as well as the base points. If $[u_{reg}]|_{C'}$ is not a constant map we are done, so assume it is
constant. Let $B'$ be the base locus supported on $C'$. By Lemma \ref{degree-splitting} and \eqref{tail} we get
$$\sum_{x\in B'}\ell(x)=\beta_{C'}(L_\theta) >\frac{1}{\epsilon}.$$ If $B'$ is a single point $x$, then $x$ violates condition $(2)$
in Definition \ref{eDef}.
Hence $B'$ must contain at least two distinct points.

If $(g,k)=(0,0)$, we have the additional case of a rational curve $C$ with no special points. The same argument will work, using now
that $\beta_{C}(L_\theta) >\frac{2}{\epsilon}$.
\end{proof}

\begin{Thm}\label{Thm3}
The stack $\QmapWe$ is a separated Deligne-Mumford stack of finite type, admitting a canonical
obstruction theory. If $W$ has at most lci singularities, then the obstruction theory is perfect. Furthermore, there is a natural morphism
$$\QmapWe\lra \mathrm{Spec}(A(W)^\G)$$
which is proper.
\end{Thm}
\begin{proof} Most of the proof of Theorem \ref{Thm1} from \S 4 carries over unchanged. First, the proof of Corollary \ref{boundedness-curve}
works for any $\epsilon$, so we obtain (see Remark \ref{general}) that the stack is algebraic and of finite type. Proposition \ref{finite autom}
is then used to conclude that the stack is Deligne-Mumford. The natural obstruction theory is the $\mu$-relative obstruction theory 
for the forgetful map
$$\mu:\QmapWe\lra\BunG, $$ which is given by
\eqref{relobstruction}. It is perfect if $W$ has at worst lci singularities (cf. Remark \ref{stablemaps}).

In fact, the only part of the argument that needs modification is the
proof of completeness in Proposition \ref{propernessProp}. Namely, when passing from 
$\widehat{C}$ to $\bar{C}$ as in that proof, we do so by contracting
only some of the trees $\Gamma_i$ in the central fiber $\widehat{C}_0$. We go briefly through
the details.

Let $(\Delta, 0)$ be a smooth pointed curve and let $\Delta ^\circ = \Delta \setminus\{0\}$. 
Let $((C, p_i), P, u)$ be a $\Delta ^\circ$-family of $\epsilon$-stable quasimaps. 
Let $B$ denote the base locus of the family. We may regard $B$ as additional sections $y_j$, and assume
that $(C,p_i,y_j)\lra\Delta^\circ$ is a family of prestable pointed curves with irreducible fibers and
nonsingular total space $C$. By shrinking $\Delta$ if necessary, we may also assume that the
length of the quasimaps in the family along each section $y_j$ is constant. Denote this length by $\ell(y_j)$.
By assumption we have $\ell(y_j)\leq 1/\epsilon$.

Let $[u]:C\setminus B\lra W/\!\!/G$ be the induced map and let $[u_{reg}]$ be its extension to $C$ (shrink $\Delta$, if needed).
The proof of Proposition
\ref{finite autom} shows that  $((C,p_i,y_j),[u_{reg}])$ is a $\Delta^\circ$-family of stable maps to $\WmodG$, of class $\beta_{reg}$. 
By Lemma \ref{degree-splitting} we have 
\begin{equation}\label{equ1}\beta(L_\theta)-\beta_{reg}(L_\theta)=\sum_j \ell(y_j).\end{equation}
Let
$$(\widehat{C},p_i,y_j)\lra\Delta,\;\;\;\; [\widehat{u}]:\widehat{C}\lra\WmodG$$
be the unique family of stable maps extending $((C,p_i,y_j),[u_{reg}])$.
For each subcurve $D$ in the central fiber $\widehat{C}_0$ we define its total degree with respect
to $L_\theta$ to be 
\begin{equation}\label{degD}\deg(D,L_\theta):=\beta_{reg}|_D(L_\theta) +\sum_{y_{j,0}\in D}\ell(y_j),\end{equation}
where $y_{j,0}$ denotes the intersection point of the section $y_j$ with $\widehat{C}_0$.

Let $\Gamma$ be a tree of rational curves in $\widehat{C}_0$ which
contains none of the markings $p_i$ and meets the rest of the curve $\overline{(\widehat{C}_0\setminus\Gamma)}$ in a single 
point $z$. Note that such a tree always has positive total degree with respect to $L_\theta$ by stability of the map $[u_{reg}]$.
We contract the tree $\Gamma$ exactly when its
total degree satisfies 
\begin{equation}\label{degGamma}\deg(\Gamma,L _{\theta})\leq\frac{1}{\epsilon}.\end{equation}

Now the exact same reasoning as in the proof of Proposition \ref{propernessProp} gives
a family of prestable quasimaps
$$((\bar{C},p_i),\bar{P},\bar{u})$$ of class $\beta$,
extending the original family over $\Delta^\circ$. We claim that the limit quasimap
is also $\epsilon$-stable.
Let $\Gamma_1,\dots , \Gamma_N$ be all trees satisfying \eqref{degGamma} that have been contracted and let $z_1,\dots z_N$
be their respective attaching points in the special fiber $\bar{C}_0$. 
Then 
\begin{equation}\label{equ2}
\beta(L_\theta)=\beta_{reg}\left|_{\overline{\widehat{C}_0\setminus\cup_l\Gamma_l}}\right.(L_\theta)+\sum_{l=1}^N\ell(z_l)
+\sum_{y_{j,0}\not\in \cup_l\Gamma_l}\ell(y_{j,0})\end{equation}
by Lemma \ref{degree-splitting} applied to the limit quasimap.

By a semicontinuity argument we have
\begin{equation}\label{equ3}
\ell(y_j)\leq \ell(y_{j,0}),\;\;\;\; \forall y_{j,0}\not\in\cup_l\Gamma_l.\end{equation}
The length at an attachment point $z_l$ is made out of contributions from two sources. First, it is clear that the restriction of 
the stable map $[u_{reg}]$ to $\Gamma_l$ contributes $\beta_{reg}|_{\Gamma_l}(L_\theta)$ to $\ell(z_l)$. Second,
each of the points $y_{j,0}\in\Gamma_l$ contributes at least $\ell(y_j)$ by semicontinuity. We deduce
\begin{equation}\label{equ4}\deg(\Gamma_l,L _{\theta})=\beta_{reg}|_{\Gamma_l}(L_\theta)+\sum_{y_{j,0}\in \Gamma_l}\ell(y_j)\leq\ell(z_l),\;\;\; l=1,\dots, N
\end{equation}
From \eqref{equ1}, \eqref{equ2}, \eqref{equ3}, and \eqref{equ4},
\begin{equation}\begin{split}
\beta(L_\theta)&=\beta_{reg}(L_\theta)+\sum_j \ell(y_j)\\
&=\beta_{reg}\left|_{\overline{\widehat{C}_0\setminus\cup_l\Gamma_l}}\right.(L_\theta)
+\sum_l\beta_{reg}|_{\Gamma_l}(L_\theta)+\sum_j \ell(y_j)\\
&\leq \beta_{reg}\left|_{\overline{\widehat{C}_0\setminus\cup_l\Gamma_l}}\right.(L_\theta) +\sum_{l=1}^N\ell(z_l)+
\sum_{y_{j,0}\not\in \cup_l\Gamma_l}\ell(y_{j,0})=\beta(L_\theta).
\end{split}\nonumber
\end{equation}
It follows that all inequalities in \eqref{equ3} and \eqref{equ4} must be equalities.
Hence, for every $1\leq l\leq N$ and every $y_{j,0}\not\in \cup_l\Gamma_l$
$$\ell(z_l)=\deg(\Gamma,L _{\theta})\leq\frac{1}{\epsilon}\;\;\mathrm{and}\;\; \ell(y_{j,0})=\ell(y_j)\leq\frac{1}{\epsilon} $$ satisfy condition $(2)$
in Definition \ref{eDef}. 
The ampleness condition $(1)$ of Definition \ref{eDef} holds by construction on the remaining rational tails
in $\bar{C}$, while on rational components with exactly two special points it is verified in the same way as
in the proof of Proposition \ref{propernessProp}.
We conclude that the limit thus constructed is an $\epsilon$-stable quasimap.
\end{proof}

Theorem \ref{Thm3} is a generalization of the corresponding results in \cite{MM1, MM2} for projective spaces and \cite{Toda} for Grassmannians.
The chamber structure for the stability parameter also extends to general targets $\WmodG$. We restrict for brevity to describing it when
$(g,k)\neq(0,0),(0,1)$ and leave the remaining cases to the reader.
We have already noted that
$$\QmapWe=\QmapW\;\; \mathrm{for}\;\; \epsilon\in\left(0,\frac{1}{\beta(L_\theta)}\right],$$
while
$$\QmapWe=\overline{M}_{g,k} (\WmodG,\beta) \;\; \mathrm{for}\;\; \epsilon\in(1,\infty).$$
It is equally easy to see that for every integer $1\leq m\leq \beta(L_\theta)-1$, the moduli space of $\epsilon$-stable quasimaps to $\WmodG$
with fixed numerical data $(g,k,\beta)$ stays constant when $\epsilon\in\left(\frac{1}{m+1},\frac{1}{m}\right]$.
The underlying curve of a quasimap parametrized by this constant moduli space is allowed only
rational tails of total degree at least $m+1$ (with respect to $L_\theta$), while the order of contact of the quasimap with the unstable locus $W^{us}$ at each base
point must be at most ${m}$.

In contrast with Toda's results for Grassmannians in \cite{Toda}, it is not true in general that 
$\epsilon$-quasimap invariants remain
unchanged when crossing a wall $\epsilon=\frac{1}{m}$. Some aspects of this are
discussed in the work \cite{CK2}, where it is proved for example that the equality still holds for toric Fano varieties,
but fails in the absence of the Fano condition. The wall-crossing contributions in the non-Fano case (and genus zero) may be viewed as being responsible
for Givental's mirror formulas, \cite{G}.
Similar interpretations may be given for $K$-nef complete intersections in toric varieties. We believe this is a general phenomenon and plan to
study it elsewhere.

%%%%%%%%%%%%%%%%%%%%%%%%%%%%%%%%%%%%%%%%%%%%%%%%%%%%%%%%%%%%%%%
\subsection{Quasimaps with one parametrized component} This version of quasimap moduli spaces
is discussed in detail
in \cite{CK} for the case of toric varieties. They are the quasimap analogues of
the so-called {\it graph spaces} in Gromov-Witten theory and
lead to extensions of the  $I$-functions
introduced by Givental to the large parameter space, i.e., to the entire Frobenius manifold given
by the Gromov-Witten theory of a target space $\WmodG$.

The quasimaps with one parametrized $\PP^1$ and the associated $I$-functions will play a
significant role in the forthcoming work \cite{CK2}, which will deal
with comparisons between quasimap and GW theories for certain classes of targets.
We restrict here to giving the definitions and a few basic facts.

Fix integers $g,k\geq 0$, a GIT quotient $\WmodG$ with $W$ affine and lci as in \S 4, and a curve class $\beta$. In addition, fix a nonsingular irreducible projective
curve $D$.

\begin{Def} \label{parametrized qmap}
A stable, $k$-pointed quasimap of genus $g$ and class $\beta$ to $\WmodG$ is specified by the data
$$( (C,p_1,\dots ,p_k), P, u, \varphi),$$
where

\begin{itemize}

\item $(C,p_1,\dots ,p_k)$ is a connected, at most nodal, projective curve of genus $g$, and $p_i$ are distinct nonsingular points of $C$,

\item $P$ is a principal $\G$-bundle on $C$,

\item $u$ is a section of the fiber bundle $$P\times_{\G}W\rightarrow C ,$$

\item $\varphi : C\lra D$ is a regular map,

 \end{itemize}
 subject to the conditions:
 \begin{enumerate}

 \item (parametrized component) $\varphi_*[C]=[D]$. Equivalently, there is a distinguished component $C_0$ of $C$ such that
 $\varphi$ restricts to an isomorphism $C_0\cong D$ and $\varphi (C\setminus C_0)$ is zero-dimensional (or empty, if $C=C_0$).

 \item (generic nondegeneracy) There is a finite (possibly empty) set of  points $B\subset C$
such that $u(C\setminus B)$ is contained in the stable locus $P\times_{\G}W^s$.

\item (prestable) The base locus $B$ is disjoint from the nodes and markings on $C$.

\item (stability) The line bundle 
$$\omega _{{C}}(p_1+\dots +p_k) \ot {\mathcal L}^\epsilon \ot \varphi^*(\omega_D^{-1}\ot{\mathcal M})$$ 
is ample for every rational $\epsilon > 0$, where
$\cL:=\cL_\theta=P\times_{\G}\CC_\theta $ and ${\mathcal M}$ is any ample line bundle
on $D$. 
(Equivalently, $\omega _{\tilde{C}}(\sum p_i+\sum q_j) \ot {\mathcal L}^\epsilon$ is ample, where
$\tilde{C}$ is the closure of $C\setminus C_0$, $p_i$ are the markings on $\tilde{C}$, and $q_j$ are the nodes $\tilde{C}\cap C_0$.)

\item the class of the quasimap $( (C,p_1,\dots ,p_k), P, u)$ is $\beta$.
 \end{enumerate}

\end{Def}

We denote by
$$\mathrm{Qmap}_{g,k}(\WmodG,\beta; D)$$
the stack parametrizing the stable quasimaps in Definition \ref{parametrized qmap}.
Note that it is empty if $g< g(D)$. However, since stability imposes no condition on the distinguished component, the
inequality $2g-2+k\geq 0$ is not required anymore.
By the same arguments as in \S 4 we obtain the following.

\begin{Thm}\label{Thm4}
If $\WmodG$ is a proper, then  $\mathrm{Qmap}_{g,k}(\WmodG,\beta; D)$
is a proper Deligne-Mumford stack of finite type, with a perfect obstruction theory.
In general, it is proper over the affine quotient $\mathrm{Spec}(A(W)^\G)$.
\end{Thm}

The most interesting case is when
$g=0$ and (necessarily) $D=\PP^1$, to which we restrict from now on.
The underlying curve for
a point in $\mathrm{Qmap}_{0,k}(\WmodG,\beta; \PP^1)$ is a $k$-pointed
tree of rational curves with one parametrized component $C_0\cong\PP^1$, that is, a
point in the stack $\widetilde{\PP^1[k]}$, the Fulton-MacPherson
space of  (not necessarily stable) configurations of $k$ distinct points on $\PP^1$. This is a smooth Artin
stack, locally of finite type (see \S 2.8 in \cite{KKO}). As in \S 1, we have a smooth Artin stack
$$\BunG\lra\widetilde{\PP^1[k]},$$ parametrizing principal $\G$-bundles on the fibers of the
universal curve over $\widetilde{\PP^1[k]}$. 

Forgetting the appropriate data, we obtain morphisms of stacks
$$\mu:\mathrm{Qmap}_{0,k}(\WmodG,\beta; \PP^1)\lra\BunG$$
and
$$\nu:\mathrm{Qmap}_{0,k}(\WmodG,\beta; \PP^1)\lra \widetilde{\PP^1[k]}.$$
The natural obstruction theory to consider is the $\mu$-relative one, given again by
\eqref{relobstruction}.

As before, if $W=V$ is a vector space, one sees easily that we have an Euler sequence
$$0\lra\mathfrak{P}\times_{\G}\mathfrak{g}\lra\mathfrak{P}\times_{\G}V\lra\cF\lra 0$$
on the universal curve
$$\pi: \cC\lra \mathrm{Qmap}_{0,k}(\WmodG,\beta; \PP^1)$$ and that the $\mu$-relative perfect obstruction theory is
given by
$$\left( R^\bullet \pi_*(\mathfrak{P}\times_{\G}V)\right )^\vee ,$$
while the $\nu$-relative perfect obstruction theory is
$$\left( R^\bullet \pi_*\cF\right )^\vee .$$

\begin{EG} The unpointed genus zero quasimap moduli spaces compactify
the $\Hom$-schemes parametrizing maps from $\PP^1$ to $\WmodG$ and
have been studied extensively earlier for several special cases of targets,
though it was not widely recognized before that the important common feature of all these examples is
that they are all GIT quotients.
\begin{itemize}

\item Let $W=\CC^N$ and  $\G=(\CC^*)^r$. Then $W/\!\!/\G$ is a toric variety
and $\mathrm{Qmap}_{0,0}(W/\!\!/\G,\beta;\PP^1)$ is the ``toric compactification" of $\mathrm{Map}_{\beta}(\PP^1,W/\!/\G)$
introduced by Givental, and by Morrison and Plesser
\cite{G}, \cite{MP}.

\item Let $W=\Hom (\CC^r,\CC^n)$ and $\G=GL(r,\CC)$ with the obvious action, linearized by the determinant.
Then $W/\!\!/\G$ is the Grassmannian $\mathrm{G}(r,n)$ of $r$-dimensional subspaces in
$\CC^n$ and $\mathrm{Qmap}_{0,0}(W/\!\!/\G,d;\PP^1)$ is the Quot scheme
parametrizing degree $d$ and rank $n-r$ quotients of $\mathcal{O}_{\PP^1}^{\oplus n}$ on $\PP^1$, first studied
in detail by Str\o mme,
\cite{S}. More generally, the ``Laumon spaces", or ``hyperquot schemes" on $\PP^1$ of \cite{L}, see also \cite{C}, \cite{Kim2} are quasimap moduli
spaces to flag varieties of type $A$, viewed as GIT quotients of a vector space by a product of $GL$'s, as described
e.g. in \cite{BCK2}.

\item Consider the vector space
$$V:=\Hom(\CC^n,\CC^n)\oplus\Hom(\CC^n,\CC^n)\oplus\Hom(\CC^r,\CC^n)
\oplus\Hom(\CC^n,\CC^r)$$ and let $\G=GL(n,\CC)$ act on the right by
$$(A,B,i,j)\cdot g=(g^{-1}Ag,g^{-1}Bg,g^{-1}i,jg),$$
with linearization given by the determinant character.

Let $W$ be the $\G$-invariant nonsingular affine subvariety given by the ADHM equation
$$[A,B]-ij=0.$$

Then $W/\!\!/\G$ is the moduli space of rank $r$ torsion-free sheaves on $\PP^2$, trivialized on the
line at infinity (in particular, for $r=1$ we get the
Hilbert scheme $\mathrm{Hilb}_n(\CC^2)$ of $n$ points in the plane), see e.g. \cite{N},
and $\mathrm{Qmap}_{0,0}(W/\!\!/\G,d;\PP^1)$ is
Diaconescu's moduli space of ADHM sheaves on $\PP^1$, see \cite{D}.

Note that $W$ is the zero locus of a section of the equivariant $\G$-bundle $\Hom(\CC^n,\CC^n)\times V$,
as discussed in Example \ref{ci}. Note also that neither $\VmodG$, nor $\WmodG$ are proper; in fact, these examples are special
cases of quiver variety targets, as discussed in \S \ref{noncompact}
\end{itemize}
\end{EG}
\begin{Rmk}\label{e-parametrized} It is immediate to extend Definition \ref{parametrized qmap} and Theorem \ref{Thm4}  to 
include all $\epsilon$-stability conditions, as in \S \ref{e-stable}. The same is
true about the material on $I$-functions discussed in \S \ref{I} below.
We leave this to the reader. 
\end{Rmk}

\subsection{The $I$-function of $W/\!\!/\G$}\label{I}
One of the most important objects in genus zero Gromov-Witten theory is the $J$-function of
a target variety. For smooth toric varieties Givental
introduced in \cite{G} their {\it small $I$-functions} via the unpointed genus zero parametrized quasimap spaces and proved
that they are related to the small $J$-functions via ``mirror transformations".
In \cite{CK},
the parametrized quasimap spaces with markings were used to define the big $I$-function of a smooth toric variety.
Here we generalize these notions to all GIT quotients $\WmodG$.

The moduli spaces $\mathrm{Qmap}_{0,k}(\WmodG,\beta; \PP^1)$ come equipped with a natural
$\CC^*$-action, induced from the usual action on $\PP^1$. The $I$-function will be defined via certain equivariant
residues for this action, so we start by describing the $\CC^*$-fixed loci in $\mathrm{Qmap}_{0,k}(\WmodG,\beta; \PP^1)$.

A stable quasimap
$$( (C,p_1,\dots ,p_k), P, u, \varphi)$$
is fixed by $\CC^*$ if and only if it satisfies the following properties:
\begin{itemize}

\item $C_0\cap \overline{(C\setminus C_0)}\subset\{ 0,\infty\}$; here $C_0$ is identified with $\PP^1$ via $\varphi$.

\item There are no markings on $C_0\setminus \{0,\infty\}$.

\item The curve class $\beta$ is ``concentrated at $0$ or $\infty$". Precisely, this means the following: there are no
base points on $C_0\setminus \{0,\infty\}$ and the resulting map
$$C_0\setminus \{0,\infty\} \lra \WmodG$$
is constant.
\end{itemize}

The absolute perfect obstruction theory
$$\mathbb{E}^\bullet=[E^{-1}\lra E^0]$$
of $\mathrm{Qmap}_{0,k}(\WmodG,\beta; \PP^1)$ is $\CC^*$-equivariant, hence each
component $F$ of the fixed point locus has an induced perfect obstruction theory, given by the $\CC^*$-fixed part
$\mathbb{E}^{\bullet, f}_F$ of the complex $\mathbb{E}^{\bullet}|_F$, and a virtual normal bundle $N_{F}^\mathrm{vir}$ given by the moving part $\mathbb{E}^{\bullet, m}_F$.

In fact, for defining the $I$-function we are only interested in the component ${F_0}$ of the fixed point locus for which the curve class $\beta$
is concentrated only at $0\in C_0$, i.e., the component parametrizing $\CC^*$-fixed quasimaps for which
$C_0\cap \overline{(C\setminus C_0)}=\{0\}$ and there is no base point or marking at $\infty$. The cases $k=0$ and $k\geq 1$ exhibit different behavior.

{\it Case $k\geq 1$}:  We have ${F_0}\cong\mathrm{Qmap}_{0,k+1}(\WmodG,\beta)$ and
$\mathbb{E}^{\bullet, f}_{F_0}$ is the usual obstruction theory of $\mathrm{Qmap}_{0,k+1}(\WmodG,\beta)$.
The (virtual) codimension of $F_0$ in
$\mathrm{Qmap}_{0,k}(\WmodG,\beta; \PP^1)$ is equal to $2$.
As explained in Lemma 7.2.7 of \cite{CK}, the
virtual normal bundle has equivariant Euler class
$$e^{\CC^*}(N_{F_0}^\mathrm{vir})=z(z-\psi),$$
where $\psi=\psi_{k+1}$ is the cotangent class at the last marking and $z$ is the equivariant parameter (i.e., $H^*_{\CC^*}(point)\cong\CC[z]$).

Associating to each $\CC^*$-fixed stable quasimap in $F_0$ the point in $\WmodG$ which is the image of the constant map
$$C_0\setminus \{0,\infty\} \lra \WmodG,$$
gives a morphism
\begin{equation}\label{eval}
F_0\lra\WmodG
\end{equation}
which is clearly just the evaluation map
$$ev_{k+1}:\mathrm{Qmap}_{0,k+1}(\WmodG,\beta)\lra\WmodG$$
at the last marking.

{\it Case $k=0$}: In this case the above identification of $F_0$ with a space of unparametrized quasimaps is not
possible, as $\mathrm{Qmap}_{0,1}(\WmodG,\beta)$ doesn't exist.

For toric varieties, when $\G=(\CC^*)^n$ is a torus and $W$ is a vector space,
$F_0$ is identified with a certain nonsingular subvariety of $\WmodG$ (depending on $\beta$).
Its virtual class is the usual fundamental class and the virtual normal bundle can be explicitly computed
using appropriate Euler sequences, see e.g. \cite{G}, \cite{CK}.
For now, we contend ourselves to remark that in the general case
the set-theoretic ``evaluation" map defined as in (\ref{eval}) can easily be seen to give a morphism
$$ev:F_0\lra \WmodG.$$ In the toric case, this is simply the embedding of $F_0$ in $\WmodG$. 
For $\WmodG$ a Grassmannian, the quasimap space with parametrized $\PP^1$ and no markings is a Quot scheme
on $\PP^1$. The fixed locus $F_0$ is identified with a disjoint union of flag bundles over the Grassmannian, while
the evaluation map is the projection. Again, the top Chern class of the virtual normal bundle, as well as its push-forward
by the evaluation map, can be calculated via appropriate Euler sequences, see \cite{BCK1}. One can extend this
description of $F_0$ and the calculation of the normal bundle in a straightforward manner to all
quotients of a vector space by a general linear group.

We are now ready to define the $I$-function. Recall from Definition \ref{eff-semigroup} that the curve classes
of quasimaps form a semigroup
$\mathrm {Eff}(W,\G,\theta)$. We let $N(W,\G,\theta)$ denote the Novikov ring of formal power series
$$N(W,\G,\theta)=\{ \sum_{\beta\in \mathrm{Eff}(W,\G,\theta)} a_\beta Q^\beta | a_\beta\in \CC\},$$
the $Q$-adic completion of the semigroup ring $\CC[\mathrm {Eff}(W,\G,\theta)]$.

\begin{Def}\label{I-fcn} The (big) $I$-function of $\WmodG$ is
\begin{align}
I_{\WmodG}({\bf t})&=1+\frac{{\bf t}}{z}+\sum_{\beta\neq 0 }Q^\beta ev_*\left(\frac{[F_0]^{\mathrm{vir}}}{e^{\CC^*}(N_{F_0}^\mathrm{vir})}\right)
\nonumber \\
&+\sum_{\beta}Q^{\beta}
\sum_{k\geq 1}\frac{1}{k!}({ev}_{k+1})_*\left (\frac{[\mathrm{Qmap}_{0,k+1}(\WmodG,\beta)]^{\mathrm{vir}}}{z(z-{\psi})}\prod_{j=1}^{k}{ev}_j^*({\bf t})\right ).\nonumber
\end{align}

It is a formal function of ${\bf t}\in H^{*}(\WmodG,\CC)$, taking values in $$H^{*}(\WmodG,\CC)\otimes_\CC N(W,\G,\theta)\left\{\!\!\left\{\frac{1}{z}\right\}\!\!\right\}.$$
\end{Def}

The above definition is an exact analogue of Givental's big $J$-function of $\WmodG$ in Gromov-Witten theory, with the graph spaces
$$G_{k,\beta}(\WmodG):=\overline{M}_{0,k}(\WmodG\times\PP^1,(\beta,1))$$ replaced by the quasimap spaces with one parametrized component.
Each of the terms in the sums is the push-forward via the appropriate evaluation map of the residue at $F_0$ of the class
$$ [\mathrm{Qmap}_{0,k}(\WmodG,\beta; \PP^1)]^{\mathrm{vir}}\cap\prod_{j=1}^{k}{ev}_j^*({\bf t}).$$
In fact, as mentioned in Remark \ref{e-parametrized}, we may define by the same formula an
$\epsilon$-$I$-function for each $\epsilon >0$, using the virtual classes of $\epsilon$-stable parametrized quasimaps. When
$\epsilon >1$ this definition reproduces precisely the usual definition of the $J$-function.

The {\it small} $I$-function is defined by restriction to the so-called small parameter space:
$$ I_{\WmodG}^{\mathrm{small}}({\bf t}^0,{\bf t}^2)=I_{\WmodG}({\bf t}_{\mathrm{small}}),$$
with
$${\bf t}_{\mathrm{small}}={\bf t}^0+{\bf t}^2\in  H^{0}(\WmodG,\CC)\oplus  H^{2}(\WmodG,\CC).$$

By analogy with \cite{G},
we also define formally {\it Givental's small} $I$-function by the formula
\begin{equation}\begin{split}
 I_{\WmodG}^{\mathrm{Givental}}({\bf t}^0,{\bf t}^2)
 &=e^{({\bf{t}}^0+{\bf{t}}^2)/z}(1+\sum_{\beta\neq 0 }Q^\beta e^{\beta({\bf t}^2)} I_{\beta}^{\mathrm{Givental}}) \\ &:=
e^{({\bf{t}}^0+{\bf{t}}^2)/z}\left(1+\sum_{\beta\neq 0 }Q^\beta e^{\beta({\bf t}^2)}  ev_*\left(\frac{[F_0]^{\mathrm{vir}}}{e^{\CC^*}(N_{F_0}^\mathrm{vir})}\right)\right).
\end{split}\end{equation}
(To avoid technicalities with the definition of $\beta({\bf t}^2)$, assume here that we have an identification between effective $\G$-equivariant
curve classes and actual effective curve classes $\beta\in H_2(\WmodG,\ZZ)$.)
When quasimap integrals of $\WmodG$ satisfy both the string and divisor equations, the two small $I$-functions coincide, however,
they will be different in general.

\end{document}